\documentclass[11pt]{article}%[10pt]{article}%\documentclass
\usepackage{amssymb,amsthm}
\usepackage{amsmath,latexsym,epsf,mathabx}
\usepackage[sort&compress,numbers]{natbib} 
\usepackage{changepage}
\usepackage[all]{xy}
\usepackage{hyperref}
\numberwithin{equation}{section}
\usepackage{color,physics}
\usepackage{tikz-cd}
\usepackage[mathlines,displaymath]{lineno}
\let\oldalign\align
\let\oldendalign\endalign

\renewenvironment{align}
{\linenomathNonumbers\oldalign}
{\oldendalign\endlinenomath}

\begin{document}
%\pagewiselinenumbers
%\linenumbers
\theoremstyle{definition}
%% ----------------------------------------------------------------------
\newcommand{\nc}{\newcommand}
\def\PP#1#2#3{{\mathrm{Pres}}^{#1}_{#2}{#3}\setcounter{equation}{0}}
\def\mr#1{{{\mathrm{#1}}}\setcounter{equation}{0}}
\def\mc#1{{{\mathcal{#1}}}\setcounter{equation}{0}}
\def\mb#1{{{\mathbb{#1}}}\setcounter{equation}{0}}
\def\Mcc{\mc{C}}
\def\Mbe{\mb{E}}
\def\Mcp{\mc{P}}
\def\Mcg{\mc{G}}
\def\fbzh{\mc{C}(-,\mc{P}(\xi))\text{-exact}}
\def\extri{(\mc{C},\mb{E},\mathfrak{s})}
\def\Gproj{\xi\text{-}\mc{G}\text{projective}}
\def\Ginj{\xi\text{-}\mc{G}\text{injective}}
\def\GP{\mc{G}\mc{P}(\xi)}
\def\GI{\mc{G}\mc{I}(\xi)}
\def\P{\mc{P}(\xi)}
\def\I{\mc{I}(\xi)}
\def\Gpd{\xi\text{-}\mc{G}\text{pd}}
\def\Extri{\mb{E}\text{-triangle}}
\def\SGP{\mc{SGP}(\xi)}
\def\nSGP{\text{$n$}\text{-}\mc{SGP}(\xi)}
\def\mSGP{\text{$m$}\text{-}\mc{SGP}(\xi)}
\def\ext{\xi \text{xt}_{\xi}}
\def\extgp{\xi \text{xt}_{\mc{GP}(\xi)}}
\def\extgi{\xi \text{xt}_{\mc{GI}(\xi)}}
\def\gext{\mc{G}\xi\text{xt}_{\xi}}
\def\nGpd{n\text{-}\xi\text{-}\mc{SG}\text{pd}}
\def\ext{\xi \text{xt}_{\xi}} 
\def\QP{\mc{QP}(\xi)}
\def\QI{\mc{QI}(\xi)}
\def\Qgpd#1{\xi\text{-}\mc{QG}\text{pd}{\left(#1\right)}}
\def\Qgid#1{\xi\text{-}\mc{QG}\text{id}{\left(#1\right)}}
\def\QGP{\mc{QGP}(\xi)}
\def\QGI{\mc{QGI}(\xi)}
%%%%%%%%FROM latexexam.tex
%\theoremstyle{definition}
\newtheorem{defn}{\bf Definition}[section]
\newtheorem{cor}[defn]{\bf Corollary}   
\newtheorem{prop}[defn]{\bf Proposition}
\newtheorem{thm}[defn]{\bf Theorem}
\newtheorem{lem}[defn]{\bf Lemma}
\newtheorem{rem}[defn]{\bf Remark}
\newtheorem{exam}[defn]{\bf Example}   
\newtheorem{fact}[defn]{\bf Fact}
\newtheorem{cond}[defn]{\bf Condition}
\def\Pf#1{{\noindent\bf Proof}.\setcounter{equation}{0}}
\def\>#1{{ $\Rightarrow$ }\setcounter{equation}{0}}
\def\<>#1{{ $\Leftrightarrow$ }\setcounter{equation}{0}}
\def\bskip#1{{ \vskip 20pt }\setcounter{equation}{0}}
\def\sskip#1{{ \vskip 5pt }\setcounter{equation}{0}}
\def\mskip#1{{ \vskip 10pt }\setcounter{equation}{0}}
\def\bg#1{\begin{#1}\setcounter{equation}{0}}
\def\ed#1{\end{#1}\setcounter{equation}{0}}

%%%%%%%%%%%%%%%%%%%%%%%%%%%%%%%%%%%%%%%%%%%%%%%%%%%%%%%%%%%%%%%%%%%%%%%%%%%%%%%%%

\title{Relative quasi-Gorensteinness in extriangulated 
categories}
%\thanks{E-mail:~zhenggang_he@163.com}

%\footnotetext{E-mail:~zhenggang_he@163.com}
\smallskip
\author{\small Zhenggang He\\
\small E-mail: zhenggang\_he@163.com\\
\small Dali Nursing Vocational College, Dali 671006, China }
\date{}
\maketitle
\baselineskip 15pt%16pt%14pt%15.5pt%\baselineskip  25.5pt %
%%%%%%%\hskip 18pt
%
% Abstract ------------------------------------------------------
%
%\begin{abstract}
%
\vskip 10pt%
\noindent {\bf Abstract}:
Let $(\Mcc,\Mbe,\mathfrak{s})$ be an extriangulated category with a proper class $\xi$ of $\Mbe$-triangles. In this paper, we study the quasi-Gorensteinness  of extriangulated categories. More precisely, we introduce the notion of  quasi-$\xi$-projective and quasi-$\xi$-Gorenstein projective  objects,  investigate some of their properties and their behavior with respect to $\Mbe$-triangles.  Moreover, we give some equivalent characterizations of objects with finite quasi-$\xi$-Gorenstein projective dimension.  As an application, our main results generalize Mashhad and  Mohammadi's work in module categories.
\mskip\

%\sskip\


%
%\end{abstract}
%\smallskip
%
\vskip 30pt
% ----------------------------------------------------------------------
%% ----------------------------------------------------------------------
%\def\baselinestretch{1}

\section{Introduction}
\qquad In the category of modules over a ring $R$, Wu and  Jans \cite{LET} introduced the notion of  quasi-projective module.   An $R$-module $M$ is called quasi-projective if and only if for every $R$-module
$A$, every epimorphism $g : M \to A$, and every morphism $f : M \to A$, there is an $f \in  \text{End}_R(M)$ such that $f = gf$. Rangaswamy and Vanaja \cite{KM} introduced the quasi-projective cover of any  object in  Abelian  categories, and they proved that the universal existence of quasi-projective covers  implies that of projective covers, provided  the Abelian  category possesses enough projectives. One can find more information about quasi-projectives and  quasi-projective cover in \cite{Baba, Fuchs,KM, LET}.

Mashhad and  Mohammadi \cite{AF} introduced quasi-Gorenstein projective and quasi-Gorenstein injective modules, and studied some of their properties and investigate their behavior with respect to short exact sequences. Also, they gave homological descriptions of quasi-Gorenstein
homological dimensions of modules when they are finite.

The notion of extriangulated categories was introduced by Nakaoka and Palu in \cite{HY} as a simultaneous generalization of exact categories and triangulated categories.  Hence many results hold on exact categories and triangulated categories can be unified in the same framework. Moreover,there exist extriangulated categories which are neither exact nor triangulated, some examples can be found in \cite{HY,Zhou}. Let $(\Mcc,\Mbe, \mathfrak{s})$ be an extriangulated category. Hu, Zhang and Zhou \cite{JDP} studied a relative homological algebra in $\Mcc$ which parallels the relative homological algebra in a triangulated category. By specifying a class of $\Mbe$-triangles, which is called a proper class $\xi$ of $\Mbe$-triangles, they introduced $\xi$-$\mathcal{G}$projective dimensions and $\xi$-$\mathcal{G}$injective dimensions and discussed their properties. Also, they  gave some characterizations of $\xi$-$\mathcal{G}$injective dimension by using derived functors in \cite{JZP}.

Since the Gorenstein homological dimension has an important role in commutative algebra,  we wish to extend the definitions of  $\xi$-$\mathcal{G}$projective and $\xi$-$\mathcal{G}$injective objects in $(\Mcc,\Mbe, \mathfrak{s})$ using quasi-$\xi$-projective and quasi-$\xi$-injective objects in this paper. We believe that this can make a valuable contribution to the development of relative homological algebra in extriangulated categories.

The paper is organized as follows. In section 2, we recall some definitions of extriangulated categories and outline some properties which will be used throughout the paper.  

In section 3, We fix a proper class of $\Mbe$-triangles $\xi$ in an extriangulated category $(\Mcc,\Mbe, \mathfrak{s})$ which has enough $\xi$-projectives and enough $\xi$-injectives satisfying Condition (WIC) (see \citep[Condition 5.8]{HY}). We  first introduce quasi-$\xi$-projective and quasi-$\xi$-injective objects in $(\Mcc,\Mbe, \mathfrak{s})$ and investigate some of their properties.  we show that the universal existence of quasi-$\xi$-projective covers (resp. quasi-$\xi$-injective envelopes ) implies that of $\xi$-projective covers (resp. $\xi$-injective envelopes ) (see Theorem \ref{LD} and Theorem \ref{HZG} ).

In section 4, we introduce quasi-$\xi$-Gorenstein projective and quasi-$\xi$-Gorenstein injective objects, and we verify some of their properties that one expects to carry over from $\xi$-Gorenstein projective  and $\xi$-Gorenstein injective objects. We show that
 the class of quasi-$\xi$-Gorenstein projective objects is $\xi$-projectively resolving  and (see Theorem \ref{Dec31} and  Theorem \ref{Jan1} ). Moreover, for the dual results, we can see   Theorem \ref{ZEH}.

 In section 4, we define quasi-$\xi$-Gorenstein projective dimension and quasi-$\xi$-Gorenstein injective dimension for any object in $(\Mcc,\Mbe, \mathfrak{s})$ , and we give homological descriptions of these dimensions (see Theorem \ref{HFY} and Theorem \ref{HSB}). 
 
 Throughout this paper, all subcategories are assumed to be full and additive.
 
\section{Preliminaries}
\qquad In this section, we briefly recall some basic definitions of extriangulated categories. We omit some details here, but the reader can find them in \cite{CZZ,JDPPR,YH,HY,HYA,JDP,JZP,JDTP,Zhou}.

\begin{defn}\citep[Definition 2.1 and 2.3]{HY}
	Suppose that $\mathcal{C}$ is equipped with an  biadditive functor $\mathbb{E}:\mathcal{C}^{op}\times \mathcal{C}\rightarrow \mathbf{Ab}$. For any pair of objects $A, C$ in $\mathcal{C}$, an element $\delta \in\mathbb{E}(C,A)$ is called an $\mathbb{E}$-extension. Thus formally, an $\mathbb{E}$-extension is a triplet $(A,\delta,C)$. Since $\mathbb{E}$ is a functor, for any $a\in \mathcal{C}(A, A')$ and $c\in \mathcal{C}(C', C)$, we have $\mathbb{E}$-extensions $\mathbb{E}(C, a)(\delta)\in \mathbb{E}(C, A') $ and $\mathbb{E}(c,A)(\delta)\in \mathbb{E}(C', A).$ We abbreviately denote them by $a_{*}\delta$ and $c^* \delta$ respectively. In this terminology, we have
	\[
		\mathbb{E}(c, a)(\delta)=c^*a_*\delta=a_*c^*\delta
	\] 
	in $\mathbb{E}(C', A')$. For any $ A$, $ C\in \mathcal{C}$, the zero element $0\in\mathbb{E}(C, A)$ is called the split $\mathbb{E}$-extension. And let $\delta\in \mathbb{E}(C,A)$ and $\delta '\in \mathbb{E}(C',A')$ be any pair of $\mathbb{E}$-extensions. A morphism $(a,c):\delta\rightarrow\delta'$ of $\mathbb{E}$-extensions is a pair of morphism $a\in \mathcal{C}(A,A')$ and $c\in \mathcal{C}(C,C')$ in $\mathcal{C}$ satisfying the equality
	\[
		a_*\delta=c^*\delta '.
	\]
\end{defn}

\begin{defn}\label{DJ}\citep[Definition 2.7]{HY}
	Let $A$, $C\in \mathcal{C}$ be any pair of objects. Two sequences of morphisms $A\stackrel{x}\longrightarrow B\stackrel{y}\longrightarrow C$ and $A\stackrel{x'}\longrightarrow B'\stackrel{y'}\longrightarrow C$ in $\mathcal{C}$ are said to be equivalent if there exists an isomorphism $b\in \mathcal{C}(B,B')$ which makes the following diagram commutative.
	\[
	\xymatrix{
		A\ar@{=}[d]\ar[r]^x &B\ar[d]^b\ar[r]^y &C\ar@{=}[d]\\
		A\ar[r]^{x'} &B'\ar[r]^{y'} &C}
	\]
\end{defn}	

We denote the equivalence class of $A\stackrel{x}\longrightarrow B\stackrel{y}\longrightarrow C$ by [$A\stackrel{x}\longrightarrow B\stackrel{y}\longrightarrow C$].

\begin{defn}\citep[Definition 2.8]{HY}
	(1) For any $A$, $C\in \mathcal{C}$, we let
	\[
		0=[A \stackrel{\begin{tiny}\begin{bmatrix}
		1 \\
		0
		\end{bmatrix}\end{tiny}}{\longrightarrow} A \oplus C \stackrel{\begin{tiny}\begin{bmatrix}
		0&1
		\end{bmatrix}\end{tiny}}{\longrightarrow} C].
	\] 
	(2) For any $[A\stackrel{x}\longrightarrow B\stackrel{y}\longrightarrow C]$ and $[A'\stackrel{x'}\longrightarrow B'\stackrel{y'}\longrightarrow C']$, we let
	\[
		[A\stackrel{x}\longrightarrow B\stackrel{y}\longrightarrow C]\oplus[A'\stackrel{x'}\longrightarrow B'\stackrel{y'}\longrightarrow C']=[A\oplus A'\stackrel{x\oplus x'}\longrightarrow B\oplus B'\stackrel{y\oplus y'}\longrightarrow C\oplus C'].
	\] 
\end{defn}

\begin{defn}\citep[Definition 2.9]{HY}
	Let $\mathfrak{s}$ be a correspondence which associates an equivalence class $\mathfrak{s}(\delta)=[A\stackrel{x}{\longrightarrow}B\stackrel{y}{\longrightarrow}C]$ to any $\mathbb{E}$-extension $\delta\in\mathbb{E}(C,A)$ . This $\mathfrak{s}$ is called a  realization of  $\mathbb{E}$, if for any morphism $(a,c):\delta\rightarrow\delta'$ with $\mathfrak{s}(\delta)=[A\stackrel{x}\longrightarrow B\stackrel{y}\longrightarrow C]$ and $\mathfrak{s}(\delta')=[A'\stackrel{x'}\longrightarrow B'\stackrel{y'}\longrightarrow C']$, there exists $b\in \mathcal{C}$ which makes the following  diagram commutative.
	\[
		\xymatrix{
		& A \ar[d]^{a} \ar[r]^{x} & B  \ar[r]^{y}\ar[d]^{b} & C \ar[d]^{c}    \\
		&A'\ar[r]^{x'} & B' \ar[r]^{y'} & C'}
	\]
In the above situation, we say that the triplet $(a,b,c)$ realizes $(a,b)$.
\end{defn}

\begin{defn}\citep[Definition 2.10]{HY}
	Let $\mathcal{C},\mathbb{E}$ be as above. A realization $\mathfrak{s}$ of $\mathbb{E}$ is said to be {\em additive} if it satisfies the following conditions.
	
	(a) For any $A,~C\in\mathcal{C}$, the split $\mathbb{E}$-extension $0\in\mathbb{E}(C,A)$ satisfies $\mathfrak{s}(0)=0$.
	
	(b) $\mathfrak{s}(\delta\oplus\delta')=\mathfrak{s}(\delta)\oplus\mathfrak{s}(\delta')$ for any pair of $\mathbb{E}$-extensions $\delta$ and $\delta'$.
	
\end{defn}

\begin{defn}\citep[Definition 2.12]{HY}
	A triplet $(\mathcal{C}, \mathbb{E},\mathfrak{s})$ is called an  extriangulated category if it satisfies the following conditions. \\
	$\rm(ET1)$ $\mathbb{E}$: $\mathcal{C}^{op}\times\mathcal{C}\rightarrow \mathbf{Ab}$ is a biadditive functor.\\
	$\rm(ET2)$ $\mathfrak{s}$ is an additive realization of $\mathbb{E}$.\\
	$\rm(ET3)$ Let $\delta\in\mathbb{E}(C,A)$ and $\delta'\in\mathbb{E}(C',A')$ be any pair of $\mathbb{E}$-extensions, realized as
	\[
		\mathfrak{s}(\delta)=[A\stackrel{x}{\longrightarrow}B\stackrel{y}{\longrightarrow}C] \quad \text{and} \quad \mathfrak{s}(\delta')=[A'\stackrel{x'}{\longrightarrow}B'\stackrel{y'}{\longrightarrow}C'].
	\]
For any pair $(a, b)$ defining such a commutative square in $\mathcal{C}$, 	
\[
	\xymatrix{
		A \ar[d]^{a} \ar[r]^{x} & B \ar[d]^{b} \ar[r]^{y} & C \\
		A'\ar[r]^{x'} &B'\ar[r]^{y'} & C'}
\]
there is  a morphism $c:C\to C'$  that the pair $(a,c)$ defines a morphism of extensions: $\delta\rightarrow\delta'$ which is realized by the triple $(a,b,c)$.\\
	$\rm(ET3)^{op}$ Dual of $\rm(ET3)$.\\
    $\rm(ET4)$ Let $\delta\in \mathbb{E}(D,A)$ and $\delta'\in \mathbb{E}(F,B)$ be $\mathbb{E}$-extensions respectively realized by
	\[
		A\stackrel{f}{\longrightarrow}B\stackrel{f'}{\longrightarrow}D\quad \text{and }\quad B\stackrel{g}{\longrightarrow}C\stackrel{g'}{\longrightarrow}F.
	\]
Then there exist an object $E\in\mathcal{C}$, a commutative diagram
	\[
		\xymatrix{
		A \ar@{=}[d]\ar[r]^{f} &B\ar[d]^{g} \ar[r]^{f'} & D\ar[d]^{d} \\
		A \ar[r]^{h} & C\ar[d]^{g'} \ar[r]^{h'} & E\ar[d]^{e} \\
		& F\ar@{=}[r] & F }
	\]
in $\mathcal{C}$, and an $\mathbb{E}$-extension $\delta''\in \mathbb{E}(E,A)$ realized by $A\stackrel{h}{\longrightarrow}C\stackrel{h'}{\longrightarrow}E$, which satisfy the following compatibilities.\\
	$(\textrm{i})$ $D\stackrel{d}{\longrightarrow}E\stackrel{e}{\longrightarrow}F$ realizes $f'_*\delta'$,\\
	$(\textrm{ii})$ $d^*\delta''=\delta$,\\
	$(\textrm{iii})$ $f_*\delta''=e^*\delta'$.\\
	$\rm(ET4)^{op}$ Dual of $\rm(ET4)$.
\end{defn}
For examples of extriangulated categories, see \citep[Example 2.13]{HY} and \citep[Remark 3.3]{JDP}. We will use the following terminology.

\begin{defn}\citep[Definition 2.15 and 2.19]{HY}
	Let $(\mathcal{C},\mathbb{E},\mathfrak{s})$ be an extriangulated category.
	
	(1) A sequence $A\stackrel{x}\longrightarrow B\stackrel{y}\longrightarrow C$ is called conflation if it realizes some $\mathbb{E}$-extension $\delta\in \mathbb{E}(C,A)$. In this case, $x$ is called an  inflation and $y$ is called a deflation.
	
	(2) If a conflation $A\stackrel{x}\longrightarrow B\stackrel{y}\longrightarrow C$ realizes $\delta\in \mathbb{E}(C,A)$, we call the pair ($A\stackrel{x}\longrightarrow B\stackrel{y}\longrightarrow C, \delta$) an $\mathbb{E}$-triangl, and write it by
	\[
		A\stackrel{x}\longrightarrow B \stackrel{y}\longrightarrow C  \stackrel{\delta}\dashrightarrow.
	\]
We usually don't write this “$\delta$” if it not used in the argument.
	
	(3) Let $A\stackrel{x}\longrightarrow B \stackrel{y}\longrightarrow C  \stackrel{\delta}\dashrightarrow$ and $A'\stackrel{x'}\longrightarrow B' \stackrel{y'}\longrightarrow C'  \stackrel{\delta'}\dashrightarrow$ be any pair of $\mathbb{E}$-triangles. If a triplet $(a,b,c)$ realizes $(a,c):\delta \rightarrow\delta'$, then we write it as
	
	\[
		\xymatrix{
			A \ar[d]^{a} \ar[r]^{x} & B  \ar[r]^{y}\ar[d]^{b} & C \ar[d]^{c} \ar@{-->}[r]^{\delta} &  \\
			A'\ar[r]^{x'} & B' \ar[r]^{y'} & C' \ar@{-->}[r]^{\delta'} &    }
	\]
	and call $(a,b,c)$ a  morphism of $\mathbb{E}$-triangles.
	
	(4) A $\mathbb{E}$-triangle $A\stackrel{x}\longrightarrow B \stackrel{y}\longrightarrow C  \stackrel{\delta}\dashrightarrow$ is called split if $\delta=0$.
\end{defn}

Following lemmas will be used many times in this paper.                              

\begin{lem}\label{FJ}\citep[Corollary 3.5]{HY}
	Assume that $(\mathcal{C},\mathbb{E},\mathfrak{s})$ satisfies $\rm(ET1)$, $\rm(ET2)$, $\rm(ET3)$ and $\rm(ET3)^{op}$. Let
	\[
		\xymatrix{
			A\ar[r]^x\ar[d]^a&B\ar[r]^y\ar[d]^b&C\ar[d]^c\ar@{-->}[r]^{\delta} & \\
			A'\ar[r]^{x'}&B'\ar[r]^{y'}&C'\ar@{-->}[r]^{\delta'} &
		}
	\]
	be any morphism of $\mathbb{E}$-triangles. Then the following are equivalent.
	
	(1) $a$ factors through $x$.
	
	(2) $a_*\delta=c^*\delta'=0$.
	
	(3) $c$ factors through $y'$.
\end{lem}	
	\noindent In particular, in the case $\delta = \delta'$ and $(a, b, c) =(1_A, 1_B, 1_C )$, we have
	\[
		x \text{\; is a section} \Leftrightarrow \delta \text{\; is split} \Leftrightarrow y \text{\;is a retraction}.
	\]

	\begin{lem}\label{BH}\citep[Proposition 3.15]{HY}
		Let $(\mathcal{C},\mathbb{E},\mathfrak{s})$ be an extriangulated category. Then the following hold.
		
		(1) Let $C$  be any object, and let $A_1\stackrel{x_1}{\longrightarrow}B_1\stackrel{y_1}{\longrightarrow}C\stackrel{\delta_1}\dashrightarrow$  and $A_2\stackrel{x_2}{\longrightarrow}B_2\stackrel{y_2}{\longrightarrow}C\stackrel{\delta_2}\dashrightarrow$  be any pair of $\mathbb{E}$-triangles. Then there is a commutative diagram in $\mathcal{C}$
	\[
		\xymatrix{
			&{A_2}\ar[d]^{m_2}\ar@{=}[r]&{A_2}\ar[d]^{x_2}\\
			{A_1}\ar@{=}[d]\ar[r]^{m_1}&M\ar[r]^{e_1}\ar[d]^{e_2}&{B_2}\ar[d]^{y_2}\\
			 {A_1}\ar[r]^{x_1}&{B_1}\ar[r]^{y_1}&C}
	\]
		\noindent which satisfies $\mathfrak{s}(y^*_2\delta_1)=[A_1\stackrel{m_1}{\longrightarrow}M\stackrel{e_1}{\longrightarrow}B_2]$, $\mathfrak{s}(y^*_1\delta_2)=[A_2\stackrel{m_2} {\longrightarrow}M\stackrel{e_2}{\longrightarrow}B_1]$ and $m_{1*}\delta_1 +m_{2*}\delta_2 = 0$.
		
		(2)  Let $A$  be any object, and let $A\stackrel{x_1}{\longrightarrow}B_1\stackrel{y_1}{\longrightarrow}C_1\stackrel{\delta_1}\dashrightarrow$  and $A\stackrel{x_2}{\longrightarrow}B_2\stackrel{y_2}{\longrightarrow}C_2\stackrel{\delta_2}\dashrightarrow$  be any pair of $\mathbb{E}$-triangles. Then there is a commutative diagram in $\mathcal{C}$
		\[
			\xymatrix{
				A\ar[d]^{x_2}\ar[r]^{x_1}&{B_1}\ar[d]^{m_2}\ar[r]^{y_1}&{C_1}\ar@{=}[d]\\
				{B_2}\ar[d]^{y_2}\ar[r]^{m_1}&M\ar[r]^{e_1}\ar[d]^{e_2}&{C_1}\\
				{C_2}\ar@{=}[r]&{C_2}
			}
		\]
		\noindent which satisfies $\mathfrak{s}(x_{2_*}\delta_1)=[B_2\stackrel{m_1}{\longrightarrow}M\stackrel{e_1}{\longrightarrow}C_1]$, $\mathfrak{s}(x_{1_*}\delta_2)=[B_1\stackrel{m_2}{\longrightarrow}M\stackrel{e_2}{\longrightarrow}C_2]$ and $e_1^*\delta_1 +e_2^*\delta_2 = 0$.
	\end{lem}
	Now we are in the position to recall the definition of proper class  of $\mathbb{E}$-triangles.  
	We always assume that $(\mathcal{C},\mathbb{E},\mathfrak{s})$ is an extriangulated categroy.

\begin{defn}\citep[Section 3]{JDP}
  Let $\xi$ be a class of $\mathbb{E}$-triangles. One says $\xi$ is closed under base change if for any $\mathbb{E}$-triangle
  \[
	 A\stackrel{x}\longrightarrow B \stackrel{y}\longrightarrow C  \stackrel{\delta}\dashrightarrow\ 
	  \in \xi
  \]
  \noindent and any morphism $c:C'\rightarrow C$, then any $\mathbb{E}$-triangle $A\stackrel{x'}\longrightarrow B' \stackrel{y'}\longrightarrow C'  \stackrel{c^*\delta}\dashrightarrow$
 belongs to $\xi$. Dually, one says $\xi$ is closed under cobase change if for any $\mathbb{E}$-triangle
  \[
	  A\stackrel{x}\longrightarrow B \stackrel{y}\longrightarrow C  \stackrel{\delta}\dashrightarrow\ 
	  \in \xi
  \]
  \noindent and any morphism $a:A\rightarrow A'$, then any $\mathbb{E}$-triangle  $A'\stackrel{x'}\longrightarrow B' \stackrel{y'}\longrightarrow C \stackrel{a_*\delta}\dashrightarrow$
   belongs to $\xi$.
\end{defn}
\begin{defn}\citep[Section 3]{JDP}
  A class of $\mathbb{E}$-triangles $\xi$ is called saturated if in the situation of Lemma \ref{BH}(1), when $A_2\stackrel{x_2}\longrightarrow B_2 \stackrel{y_2}\longrightarrow C  \stackrel{\delta_2}\dashrightarrow$ and $A_1\stackrel{m_1}\longrightarrow M \stackrel{e_1}\longrightarrow B_2  \stackrel{y_2^*\delta_1}\dashrightarrow$ belong to $\xi$, then the
  $\mathbb{E}$-triangle $A_1\stackrel{x_1}\longrightarrow B_1 \stackrel{y_1}\longrightarrow C  \stackrel{\delta_1}\dashrightarrow$ belongs to  $\xi$. 
\end{defn}

We denote the calss  of an extriangulated category $(\mathcal{C},\mathbb{E},\mathfrak{s})$ by $\Delta_0$ which consisting of the split $\mathbb{E}$-triangle. 
\begin{defn}\label{ZL}\citep[Definition 3.1]{JDP}
  Let $\xi$ be a class of $\mathbb{E}$-triangles which is closed under isomorphisms. $\xi$ is called a  proper class of $\mathbb{E}$-triangles if the following conditions holds:
  
  (1) $\xi$ is closed under finite coproducts and $\Delta_0 \subseteq \xi$.
  
  (2) $\xi$ is closed under base change and cobase change.
  
  (3) $\xi$ is saturated.
\end{defn}

\begin{defn}\citep[Definition 3.4]{JDP}
	Let $\xi$ be a proper class of $\mathbb{E}$-triangles. A morphism $x$ is called $\xi$-inflation if there exists an $\mathbb{E}$-triangle $A\stackrel{x}\longrightarrow B \stackrel{y}\longrightarrow C  \stackrel{\delta}\dashrightarrow$ in $\xi$. A morphism $y$ is called $\xi$-deflation if there exists an $\mathbb{E}$-triangle $A\stackrel{x}\longrightarrow B \stackrel{y}\longrightarrow C  \stackrel{\delta}\dashrightarrow$ in $\xi$.
\end{defn}

\begin{defn}\citep[Definition 4.1]{JDP} An object $P\in \mathcal{C}$ is called $\xi$-projective if for any $\mathbb{E}$-triangle
  \[
	A\stackrel{x}\longrightarrow B \stackrel{y}\longrightarrow C  \stackrel{\delta}\dashrightarrow
  \]
  lies in $\xi$, the induced sequence of abelian groups
  \[
  0\longrightarrow\mathcal{C}(P,A)\longrightarrow \mathcal{C}(P,B)\longrightarrow\mathcal{C}(P,C)\longrightarrow 0
  \]
is exact. Dually, we have the definition of $\xi$-injective.
\end{defn}

We denote $\P$ (resp. $\mathcal{I}(\xi)$) the class of $\xi$-projective (resp. $\xi$-injective) objects of $\Mcc$. An extriangulated category $(\mathcal{C},\mathbb{E},\mathfrak{s})$ is said to have enough $\xi$-projectives  (resp. enough $\xi$-injectives) provided that for each object $A\in \mathcal{C}$ there exists an $\mathbb{E}$-triangle $K\longrightarrow P\longrightarrow A\dashrightarrow$ (resp. $A\longrightarrow I\longrightarrow K\dashrightarrow$) in $\xi$ with $P\in\mathcal{P}(\xi)$(resp. $I\in\mathcal{I}(\xi)$). One can also find the above concept in \citep[Section 4]{JDP}.

\begin{lem}\label{FBZH1}\citep[Lemma 4.2]{JDP}
  If $\mathcal{C}$ has enough $\xi$-projectives, then an $\mathbb{E}$-triangle \[ A\longrightarrow B\longrightarrow C\dashrightarrow\] lies in $\xi$ if and only if induced sequence of abelian groups
  \[
  0\longrightarrow\mathcal{C}(P,A)\longrightarrow \mathcal{C}(P,B)\longrightarrow\mathcal{C}(P,C)\longrightarrow0
  \]
  is exact for all $P\in\mathcal{P}(\xi)$.
\end{lem}

\begin{defn}\citep[Definition 4.5]{JDP}
	Let $\mathcal{W}$ be a class of objects in $\Mcc$. An $\mathbb{E}$-triangle $A\longrightarrow B\longrightarrow C\dashrightarrow$ in $\mathcal{C}$ is called to be $\Mcc(-,\mathcal{W})$-exact (respectively $\Mcc(\mathcal{W},-)$-exact) if for any $W\in\mathcal{W}$, the induced sequence of abelian group $0 \rightarrow \Mcc(C,W)\rightarrow\Mcc(B,W)\rightarrow\Mcc(A,W)\rightarrow0$ (respectively $0 \rightarrow \Mcc(W,A)\rightarrow\Mcc(W,B)\rightarrow\Mcc(W,C)\rightarrow0$) is exact in $\mathbf{Ab}$.
\end{defn}

Here we introduce the weak idempotent completeness condition (WIC) for an extriangulated categories. 
\begin{cond}\citep[Condition 5.81]{HY} (WIC Condition)\label{WIC}
  Consider the following conditions.
  
  (1) Let $f\in\Mcc(A,B),g\in\Mcc(B,C)$ be any pair of morphisms. If $gf$ is an inflation, then so is $f$.
  
   (2) Let $f\in\Mcc(A,B),g\in\Mcc(B,C)$ be any pair of morphisms. If $gf$ is a deflation, then so is $g$.
\end{cond}

Throughout the  paper, we always assume that $\mathcal{C}=(\Mcc,\Mbe,\mathfrak{s})$ is an extriangulated category with enough $\xi$-projectives and enough $\xi$-injectives satisfying  Condition(WIC) and $\xi$ is a proper class of $\Mbe$-triangles in $(\Mcc,\Mbe,\mathfrak{s})$.

Now we recall some definitions and properties about relative derived functors  for extriangulated categories.  And for more  details reader can see \cite{JZP}.

A $\xi$-projective resolution of an object $A\in\Mcc$ is a  right bounded   $\xi$-exact complex
	\[
	\xymatrix{
		\cdots\ar[r]&P_{n}\ar[r]&P_{n-1}\ar[r]&{\cdots}\ar[r]&P_{1}\ar[r]&P_{0}\ar[r]&A\ar[r]&0
	}
	\]
in $\Mcc$  with $P_{n}\in\P$ for all $n\geqslant0$. Dually, 
a $\xi$-injective coresolution of an object $A\in\Mcc$ is a  left bounded  $\xi$-exact complex
	\[
	\xymatrix{
		0\ar[r]& A\ar[r]&I_{0}\ar[r]&I_{-1}\ar[r]&{\cdots}\ar[r]&I_{-(n-1)}\ar[r]&I_{-n}\ar[r]&\cdots
	}
	\]
in $\Mcc$  with $I_{n}\in \mathcal{I}(\xi)$ for all $n\geqslant 0$.

\begin{defn}\citep[Definition 3.2]{JZP}
	Let $A$ and $B$ be objects in $\mathcal{C}$.
	
	(1) If we choose a $\xi$-projective resolution $\xymatrix{\mathbf{P}\ar[r]&A}$ of $A$, then for any  integer $n\geqslant 0$, the $\xi$-cohomology groups $\xi \text{xt}_{\mathcal{P}(\xi)}^n(A, B)$ are defined as
	\[
	\xi \text{xt}_{\mathcal{P}(\xi)}^n(A, B)=H^n(\mathcal{C}(\mathbf{P},B)).
	\]
	
	(2) If we choose a $\xi$-injective coresolution $\xymatrix{B\ar[r]&\mathbf{I}}$ of $B$, then for any  integer $n\geqslant 0$, the $\xi$-cohomology groups $\xi \text{xt}_{\mathcal{I}(\xi)}^n(A, B)$ are defined as
	\[
	\xi \text{xt}_{\mathcal{I}(\xi)}^n(A, B)=H^n(\mathcal{C}(A,\mathbf{I})).
	\]
	\end{defn}
There exists an isomorphism $\xi \text{xt}_{\mathcal{P}(\xi)}^n(A, B)\backsimeq \xi \text{xt}_{\mathcal{I}(\xi)}^n(A, B)$, and denoting the isomorphism class of this abelian group by $\xi \text{xt}_{\xi}^n(A,B)$. Let $\mathcal{X}$ be a class of some object in $\Mcc$, and  we set 
\[\mathcal{X}^{\bot}=\{B\in\Mcc~|~\ext^n(X,B)=0,\forall n\geqslant1,\forall X\in\mathcal{X}\}.\]

It is clear by definition, there is always a natural isomorphism 
\[\varphi:\mathcal{C}(A,B)\rightarrow\ext^0(A,B),\]
when  $A\in \P$ or $B\in \mathcal{I}(\xi)$. 

\begin{lem}\label{LZHL}\citep[Lemma 3.4]{JZP}
	If $A\stackrel{x}\longrightarrow B \stackrel{y}\longrightarrow C  \stackrel{\delta}\dashrightarrow$ is an $\mathbb{E}$-triangle in $\xi$, then for any objects $X$  in $\mathcal{C}$, we have the following long exact sequences in $\mathbf{Ab}$
	\[
	\xymatrix{
	0\ar[r]&\ext^0(X,A)\ar[r]&\ext^0(X,B)\ar[r]&\ext^0(X,C)\ar[r]&\ext^1(X,A)\ar[r]&\cdots
	}
	\]
	and
	\[
	\xymatrix{
	0\ar[r]&{\ext^0(C,X)}\ar[r]&{\ext^0(B,X)}\ar[r]&{\ext^0(A,X)}\ar[r]&{\ext^1(C,X)}\ar[r]&{\cdots}}
	\]
\end{lem}

\section{Quasi-$\xi$-projective objects and some fundamental properties}

In this section, we introduce the notion of the quasi-$\xi$-projective objects and examine the properties of this objects.
\subsection{The definition of quasi-$\xi$-projective objects}
\begin{defn}
	An object $Q$ in $\Mcc$ is called quasi-$\xi$-projective if for any $\mathbb{E}$-triangle 
	\[
		A\stackrel{x}\longrightarrow Q \stackrel{y}\longrightarrow B  \stackrel{\delta}\dashrightarrow
  \]
	lies in $\xi$, the induced sequence of abelian groups
	\[
	0\longrightarrow\mathcal{C}(Q,A)\longrightarrow \mathcal{C}(Q,Q)\longrightarrow\mathcal{C}(Q,B)\longrightarrow 0
	\]
  is exact.  Dually, we have the definition of quasi-$\xi$-injective.  We denote $\QP$ (resp. $\QI$) the  full subcategory  of quasi-$\xi$-projective (resp. quasi-$\xi$-injective) objects of $\Mcc$. 
\end{defn}

\begin{rem}
	According to the definition, we have 

(1) $\P\subseteq \QP$.

(2) $\QP$ is closed under the isomorphism.
\end{rem}
\begin{proof}
	(1) is obvious. We only need to prove (2). Let $Q$ lie in $\QP$ and $f:Q'\to Q$ be  an isomorphism. For any $\mathbb{E}$-triangle $A\stackrel{x'}\longrightarrow Q'\stackrel{y'}\longrightarrow B\stackrel{\delta}\dashrightarrow$ in $\xi$, there is a commutative diagram as follows:
	\[
	\xymatrix{
		A\ar@{=}[d]\ar[r]^{x'} &Q'\ar[d]^f\ar[r]^{y'} &B\ar@{=}[d]\\
		A\ar[r]_{fx'} &Q\ar[r]_{y'f^{-1}} &B}
	\]
So $A\stackrel{fx'}\longrightarrow Q\stackrel{y'f^{-1}}\longrightarrow B\stackrel{\delta}\dashrightarrow$ is an $\mathbb{E}$-triangle in $\xi$ by Definition \ref{DJ}. Since $Q$ is quasi-$\xi$-projective, we obtain a short exact sequence
\[
	0\longrightarrow\Mcc(Q,A)\longrightarrow\Mcc(Q,Q)\longrightarrow\Mcc(Q,B)\longrightarrow 0.
\]
For any  morphism $a\in \Mcc(Q',B)$, note  that $af^{-1}$ belongs to $\Mcc(Q,B)$. Thus there is a morphism $b\in\Mcc(Q,Q)$ such that $af^{-1}=y'f^{-1}b$. So we have \[a=y'(f^{-1}bf)=\Mcc(Q',y')(f^{-1}bf).\] This implies that $\Mcc(Q',y')$ is an epimorphism. For any  morphism $c\in \Mcc(Q',A)$ satisfies $\Mcc(Q',x')(c)=x'c=0$. Note that\[ \Mcc(Q,fx')(cf^{-1})=fx'cf^{-1}=0,\]
then $cf^{-1}=0$ since $ \Mcc(Q,fx')$ is a monomorphism. This implies  $c=0$ because $f^{-1}$ is an isomorphism. So $\Mcc(Q',x')$ is also a monomorphism. Moreover $A\stackrel{x'}\longrightarrow Q'\stackrel{y'}\longrightarrow B\stackrel{\delta}\dashrightarrow$  is an $\mathbb{E}$-triangle, then the sequence $(*)$ is exact at $\Mcc(Q',Q')$. Therefore, the sequence $(*)$
\[
	0\rightarrow\Mcc(Q',A)\rightarrow\Mcc(Q',Q')\rightarrow\Mcc(Q',B)\rightarrow 0\quad (*)
\]
is exact. It is enough to say $Q'$ lies in $\QP$.

\end{proof}
\subsection{Some fundamental properties of quasi-$\xi$-projective objects}
The following lemma gives a condition under which an object becomes $\xi$-projective.

\begin{lem}\label{TS}
	 Let $A$ be an object in $\Mcc$, then $A$ is $\xi$-projective if and only if there exists an $\xi$-deflation $p:P\to A$ with $P\in\P$ and $A\oplus P\in\QP$.
\end{lem}

\begin{proof}
	We only prove the ``if'' prat. Since  $p:P\to A$ is a $\xi$-deflation, there exists an $\mathbb{E}$-triangle $ A'\stackrel{x}\longrightarrow P \stackrel{p}\longrightarrow A  \stackrel{\delta}\dashrightarrow$  in $\xi$. Note that the split $\mathbb{E}$-triangle $A \stackrel{\begin{tiny}\begin{bmatrix}
		1 \\
		0
\end{bmatrix}\end{tiny}}{\longrightarrow} A \oplus P \stackrel{\begin{tiny}\begin{bmatrix}
		0&1\end{bmatrix}\end{tiny}}{\longrightarrow} P\dashrightarrow$ is in $\xi$, then the composite morphsim $p\begin{small}\begin{bmatrix}0&1\end{bmatrix}\end{small}=\begin{small}\begin{bmatrix}
			0&p\end{bmatrix}\end{small}$ is  a $\xi$-deflation by \citep[Corollary 3.5]{JDP}. So there is an  $\mathbb{E}$-triangle $X \stackrel{}{\longrightarrow} A \oplus P \stackrel{\begin{tiny}\begin{bmatrix}
		0&p\end{bmatrix}\end{tiny}}{\longrightarrow} P\dashrightarrow$  in $\xi$ by definition of $\xi$-deflation. There exists a morphism $f\in\Mcc(A\oplus P,A\oplus P)$ such that the diagram 
\[
	\xymatrix{
		&&A\oplus P\ar[d]^{\begin{tiny}\begin{bmatrix}
			0&1\end{bmatrix}\end{tiny}}\ar[dl]_f\\
		X\ar[r]&A\oplus P\ar[r]_{\begin{tiny}\begin{bmatrix}
			0&p\end{bmatrix}\end{tiny}}&A\ar@{-->}[r] & 
	}
\]
commutes by the quasi-$\xi$-projectivity of $A\oplus P$. Then we have 
\[
1_A=\begin{tiny}\begin{bmatrix}1_A&0\end{bmatrix}\end{tiny}\begin{tiny}\begin{bmatrix}
	1_A \\0\end{bmatrix}\end{tiny}=\begin{tiny}\begin{bmatrix}0&p\end{bmatrix}\end{tiny}f\begin{tiny}\begin{bmatrix}
		1_A \\0\end{bmatrix}\end{tiny}=p\qty(\begin{tiny}\begin{bmatrix}0&1\end{bmatrix}\end{tiny}f\begin{tiny}\begin{bmatrix}
			1 \\0\end{bmatrix}\end{tiny}).
\]
Thus $p$ is a retraction and the $\mathbb{E}$-triangle $A'\stackrel{x}{\longrightarrow} P \stackrel{p}{\longrightarrow} A\dashrightarrow$ is split. So $A\oplus A'\backsimeq  P$. Since  subcategory $\P$ closed under the direct summands, $A$ is $\xi$-projective.
\end{proof}

\begin{defn}
 A $\xi$-deflation $f:B\to C$ is said to be essential if for any morphsim $g:A\to B$ is a $\xi$-deflation provided that the composite morphsim $fg$ is a $\xi$-deflation. 
\end{defn}
 
Now we introduce the notion of the $\xi$-projective cover and quasi-$\xi$-projective cover.

\begin{defn}
Let  $f:X\to A$ be a $\xi$-deflation in $\mathcal{C}$.

	(1) The morphism $f$ is called a $\xi$-projective cover of $A$ if $X$ is $\xi$-projective  and $f$ is a  essential  $\xi$-deflation.

	(2) The morphism $f$ is called a quasi-$\xi$-projective cover of $A$ if $X$ is quasi-$\xi$-projective and $f$ is a  essential  $\xi$-deflation.
\end{defn}
Dually, we can define the $\xi$-injective envelope and the quasi-$\xi$-injective envelope of any object in $\mathcal{C}$.

\begin{prop}
	Let $P$ be a $\xi$-projective object,then the following are equivalent for a $\xi$-deflation $f:P\to A$.

	(1) The morphism $f$ is a $\xi$-projective cover of $A$.

	(2) Any endomorphism $x:P\to P$ satisfying $fx=f$ is an isomorphism.
\end{prop}

\begin{proof}
	(1) $\Rightarrow $ (2): Let $x:P\to P$ be an endomorphism satisfying $fx=f$, then $x$ is $\xi$-deflation since  $f$ is  essential. Thus there exists an $\mathbb{E}$-triangle  $X\longrightarrow P \stackrel{x}{\longrightarrow} P\dashrightarrow$ in $\xi$. Since $P\in\P$, we obtain a short exact sequence
		\[
  0\longrightarrow\mathcal{C}(P,X)\longrightarrow \mathcal{C}(P,P)\stackrel{\mathcal{C}(P,x)}\longrightarrow\mathcal{C}(P,P)\longrightarrow0.
  \]
  So there is a morphism $x'\in\mathcal{C}(P,P)$ such that $xx' = 1_P$. Note that  $f$ is essential and $fx'=(fx)x'=f(xx')=f$, therefore $x'$ is a $\xi$-deflation. Similarly,  there is a morphism $x''\in\mathcal{C}(P,P)$ such that $x'x'' = 1_P$. Since
\[
x=x(x'x'')=(xx')x''=x'',
\]
we obtain $xx'=x'x=1_P$. This implies that $x$ is an isomorphism.

(2) $\Rightarrow $ (1): Let $x:P\to P$ be a morphism such that $fx$ is a $\xi$-deflation, then there exists an $\mathbb{E}$-triangle  $Y\longrightarrow P \stackrel{fx}{\longrightarrow} A\dashrightarrow$ in $\xi$. Recall that $P$ is $\xi$-projective, we obtain a exact sequence
\[
0\longrightarrow\mathcal{C}(P,Y)\longrightarrow \mathcal{C}(P,P)\stackrel{\mathcal{C}(P,fx)}\longrightarrow\mathcal{C}(P,A)\longrightarrow0.
\]
There is a morphism $x'\in\mathcal{C}(P,P)$ such that $(fx)x' = f$ because of $f\in\mathcal{C}(P,A)$. Thus the composite morphism $xx'$ is  an isomorphism. So there exists a morphism $x''\in\mathcal{C}(P,P)$ suth that $(xx')x''=1_P$. Since $1_P$ is a $\xi$-deflation, $xx'$ and  $x$ are both $\xi$-deflations by \citep[Proposition 4.13(2)]{JDP}. This is enough to say $f$ is a $\xi$-projective cover of $A$.
\end{proof}

\begin{cor}
	Let $f:P\to A$ and $f':P'\to A$ be two $\xi$-projective covers of an object $A$, then there is an isomorphism $\varphi: P\to P'$ suth that $f=f'\varphi$.
\end{cor}

\begin{rem}
The $\xi$-projective cover of any object in extriangulated category $(\mathcal{C},\mathbb{E},\mathfrak{s})$  is  unique up to isomorphism.
\end{rem}

\begin{prop}
Let $P$ is a $\xi$-projective object and $f:Q\to P$ is a quasi-$\xi$-projective cover of $P$, then $Q$ is isomorphism to $P$. 
\end{prop}

\begin{proof}
 There is an $\mathbb{E}$-triangle  $X\longrightarrow Q \stackrel{f}{\longrightarrow} P\dashrightarrow$ in $\xi$ since the morphism  $f:Q\to P$ is a quasi-$\xi$-projective cover of $P$. Note that $P$ is $\xi$-projective, So there exists a exact sequence
\[
0\longrightarrow\mathcal{C}(P,X)\longrightarrow \mathcal{C}(P,Q)\stackrel{\mathcal{C}(P,f)}\longrightarrow\mathcal{C}(P,P)\longrightarrow0
\]
in $\mathbf{Ab}$. Thus we obtain a morphism $f'\in\mathcal{C}(P,Q)$ suth that $ff'=1_P$. Obviously, the morphism $1_P$ is a $\xi$-deflation. So the morphism  $f'$ is a $\xi$-deflation by the essentiality of $f$. Then there is an  $\mathbb{E}$-triangle  $X'\longrightarrow P \stackrel{f'}{\longrightarrow} Q\dashrightarrow$ in $\xi$. Since the class of $\xi$-deflation is closed under composition (see \citep[Corollary 3.5]{JDP}), the composite morphsim $f'f\in\mathcal{C}(Q,Q)$ is a $\xi$-deflation, i.e. there is an  $\mathbb{E}$-triangle  $Y\longrightarrow Q \stackrel{f'f}{\longrightarrow} Q\dashrightarrow$ in $\xi$. Note that $Q$ is quasi-$\xi$-projective, there exists a exact sequence
\[
0\longrightarrow\mathcal{C}(Q,X)\longrightarrow \mathcal{C}(Q,Q)\longrightarrow\mathcal{C}(Q,Q)\longrightarrow0.
\]
So we can obtain a morphism $g\in\mathcal{C}(Q,Q)$ such that $(f'f)g=1_Q$, i.e. there is the following commutative diagram in $\mathcal{C}$:
\[
	\xymatrix{
		&&Q\ar[d]^{1_Q}\ar[dl]_{g}\\
		Y\ar[r]&Q\ar[r]_{f'f}&Q\ar@{-->}[r] & 
	}
\]
Similarly,there is a morphisms $x\in\mathcal{C}(Q,Q)$ such that $g=(f'f)x$, 
thus
\[
1_Q=(f'f)g=(f'f)(f'f)x=(f'f)x=g,
\]
it implies that $f'f=1_Q$. This is enough to say $f:Q\to P$ is an isomorphism, so $Q\simeq P$.
\end{proof}
\begin{rem}
	The quasi-$\xi$-projective cover of an $\xi$-projective object in extriangulated category $(\mathcal{C},\mathbb{E},\mathfrak{s})$  is  unique up to isomorphism.
\end{rem}

\begin{lem}\label{Iso}
Let $f:A\to B$ is an isomorphism, then $f$ is both a $\xi$-inflation and a $\xi$-deflation.
\end{lem}

\begin{proof}
	It is easy to check that the following are two commutative diagram.
\[
		\xymatrix{
			A  \ar@{=}[d] \ar[r]^{1} & A  \ar[r]\ar[d]^{f^{-1}} & 0 \ar@{=}[d]  &&& 0  \ar@{=}[d] \ar[r] & B  \ar[r]^{1}\ar[d]^{f^{-1}} & B \ar@{=}[d]\\
			A\ar[r]^{f} & B\ar[r] & 0 &&&  0\ar[r] & A\ar[r]^{f}& B }
	\] 
Hence $A\stackrel{f}{\longrightarrow} B \longrightarrow 0\dashrightarrow$ and  $0\longrightarrow A\stackrel{f}{\longrightarrow} B\dashrightarrow$ are two split $\mathbb{E}$-triangles. So $f$ is both a $\xi$-inflation and a $\xi$- deflation.
\end{proof}

\begin{prop}\label{FDe}
	Let $A\stackrel{x}{\longrightarrow} B \stackrel{y}{\longrightarrow} C\stackrel{\delta}{\dashrightarrow}$ and  $A'\stackrel{x'}{\longrightarrow} B' \stackrel{y'}{\longrightarrow} C\stackrel{\delta'}{\dashrightarrow}$ be two $\mathbb{E}$-triangles in $\mathcal{C}$ where $x'$ is a monomorphism. And suppose that the following  diagram of  $\mathbb{E}$-triangles with $a_{*}\delta=\delta'$
	\[
		\xymatrix{
			A \ar[d]^{a} \ar[r]^{x} & B  \ar[r]^{y}\ar[d]^{b} & C \ar@{=}[d] \ar@{-->}[r]^{\delta} &  \\
			A'\ar[r]^{x'} & B' \ar[r]^{y'} & C' \ar@{-->}[r]^{\delta'} &    }
	\] 
is commutative.   If  $b$ is a deflation, then so is $a$.
\end{prop}

\begin{proof}
If  $b$ is a deflation, then there is an  $\mathbb{E}$-triangle	 $X\longrightarrow B \stackrel{b}{\longrightarrow} B'\stackrel{\kappa }{\dashrightarrow}$ in $\mathcal{C}$. By $\rm(ET4)^{op}$, there exists following commutative diagram with $\mathbb{E}$-triangles and $bf=x'y.$
\[
		\xymatrix{
			X\ar@{=}[d] \ar[r] & D  \ar[r]^{y}\ar[d]^{f} & A' \ar[d]^{x'} \ar@{-->}[r] &  \\
			X\ar[r] & B\ar[d]^{g} \ar[r]^{b} & B'\ar[d]^{y'} \ar@{-->}[r]^{\kappa} & \\
			& C\ar@{=}[r]\ar@{-->}[d]^{\theta}&C\ar@{-->}[d]^{\delta'} \\
			& & &  }
	\] 
Since $g=y'b=y$, there is a morphism $h\in\mathcal{C}(A,D)$ such that the triplet $(a,1,1)$ realizes $(h,1):\delta\to \theta$ as in following commutative diagram by $\rm(ET3)^{op}$. 
\[
		\xymatrix{
			A \ar@{-->}[d]^{h} \ar[r]^{x} & B  \ar[r]^{y}\ar@{=}[d] & C \ar@{=}[d] \ar@{-->}[r]^{\delta} &  \\
			D\ar[r]^{f} & B \ar[r]^{g} & C \ar@{-->}[r]^{\theta} & }
	\] 
So $h$ is an isomorphism by \citep[Corollary 3.6]{HY}. Note that \[x'a=bx=bfh=x'yh\] and $x'$ is a monomorphism, thus we have $a=yh$. It is enough to say $a$ is a deflation by Lemma \ref{Iso} and \citep[Remark 2.16]{HY}.
\end{proof}

\begin{lem}\label{MDe}
	Let $P$ be a $\xi$-projective object  and $f:A\to A'$ be a $\xi$-deflation in $\mathcal{C}$. If there is a following commutative diagram with two split $\mathbb{E}$-triangles 
	\[
		\xymatrix{
			A \ar[d]^{f} \ar[r]^{\begin{tiny}\begin{bmatrix}
				1\\0
				\end{bmatrix}\end{tiny}\quad} & A\oplus P  \ar[r]^{\quad \begin{tiny}\begin{bmatrix}
					0&1
					\end{bmatrix}\end{tiny}}\ar[d]^{\begin{tiny}\begin{bmatrix}
						f & x\\0 & 1
						\end{bmatrix}\end{tiny}} & P \ar@{=}[d] \ar@{-->}[r]^{0} &  \\
			A'\ar[r]_{\begin{tiny}\begin{bmatrix}
				1\\0
				\end{bmatrix}\end{tiny}\quad} & A'\oplus P \ar[r]_{\quad\begin{tiny}\begin{bmatrix}
					0&1
					\end{bmatrix}\end{tiny}} & P \ar@{-->}[r]^{0} & }
	\] 
then the morphism $\begin{tiny}\begin{bmatrix}
	f & x\\0 & 1
	\end{bmatrix}\end{tiny}\in\mathcal{C}(A\oplus P,A'\oplus P)$ is also a $\xi$-deflation in $\mathcal{C}$.
\end{lem}
	
\begin{proof}
	Note that  $f$ is a $\xi$-deflation, then there exists an $\mathbb{E}$-triangle $X\stackrel{g}{\longrightarrow} A \stackrel{f}{\longrightarrow} A'\stackrel{\delta}{\dashrightarrow}$ in $\xi$. Thus  $X\oplus 0\stackrel{\begin{tiny}\begin{bmatrix}
		g & 0\\ 0 & 0\end{bmatrix}\end{tiny}}{\longrightarrow} A\oplus P \stackrel{\begin{tiny}\begin{bmatrix}
			f & 0\\ 0 & 1\end{bmatrix}\end{tiny}}{\longrightarrow} A'\oplus P \stackrel{\delta\oplus 0}{\dashrightarrow}$ is a  $\mathbb{E}$-triangle in $\xi$, since  $\xi$ is closed under finite coproducts. 	Since $P$ lies in $\P$, we have a exact sequence like
	\[0\longrightarrow\Mcc((P,X)\stackrel{\Mcc(P,g)}\longrightarrow \Mcc((P,A)\stackrel{\Mcc((P,f)}\longrightarrow\Mcc((P,A')\longrightarrow0.\]
	Therefore, for any morphism $x\in\mathcal{C}(P,A')$, there is a morphism $y\in\mathcal{C}(P,A)$  such that $-x=fy$. Consider the following diagram
\[
		\xymatrix{
			X\oplus 0 \ar@{=}[d]  \ar[r]^{\quad\begin{tiny}\begin{bmatrix}
				g & 0\\ 0 & 0\end{bmatrix}\end{tiny}\quad} & A\oplus P  \ar[r]^{\begin{tiny}\begin{bmatrix}
					f & 0\\ 0 & 1\end{bmatrix}\end{tiny}}\ar[d]^{\begin{tiny}\begin{bmatrix}
						1 & y\\0 & 1
						\end{bmatrix}\end{tiny}} & A'\oplus P \ar@{=}[d]   \\
			X\oplus 0 \ar[r]_{\quad\begin{tiny}\begin{bmatrix}
				g & 0\\ 0 & 0\end{bmatrix}\end{tiny}\quad} & A\oplus P \ar[r]_{\begin{tiny}\begin{bmatrix}
					f & x\\0 & 1
					\end{bmatrix}\end{tiny}} & A'\oplus P }
	\] 
Obviously, it is commutative and the morphsim $\begin{tiny}\begin{bmatrix}1 & y\\0 & 1\end{bmatrix}\end{tiny}$ is an isomorphism which implies that $X\oplus 0\stackrel{\begin{tiny}\begin{bmatrix}
		g & 0\\ 0 & 0\end{bmatrix}\end{tiny}}{\longrightarrow} A\oplus P \stackrel{\begin{tiny}\begin{bmatrix}
			f & x\\ 0 & 1\end{bmatrix}\end{tiny}}{\longrightarrow} A'\oplus P \stackrel{\delta\oplus 0}{\dashrightarrow}$ is an $\mathbb{E}$-triangle in $\xi$. So $\begin{tiny}\begin{bmatrix}
	f & x\\0 & 1
	\end{bmatrix}\end{tiny}\in\mathcal{C}(A\oplus P,A'\oplus P)$ is a $\xi$-deflation. 
	\end{proof}

\begin{thm}\label{LD}
Every object in $\mathcal{C}$ possesses a $\xi$-projective cover if and only if  every object in $\mathcal{C}$ possesses a quasi-$\xi$-projective cover.
\end{thm}

\begin{proof}
	We only need to prove the ``if'' part since $\P\subseteq\QP$. For any object $A\in\mathcal{C}$, note that $\mathcal{C}$ has enough $\xi$-projective objects, there is an $\mathbb{E}$-triangle $X\stackrel{x}{\longrightarrow} P \stackrel{y}{\longrightarrow} A\stackrel{\delta}{\dashrightarrow}$ in $\xi$ with $P\in\P$. Note that  $A\oplus P$ has a quasi-$\xi$-projective cover. Thus we can obtain an $\mathbb{E}$-triangle $Y\stackrel{a}{\longrightarrow} Q \stackrel{\begin{tiny}\begin{bmatrix}
		b \\  c\end{bmatrix}\end{tiny}}{\longrightarrow} A\oplus P\stackrel{\theta}{\dashrightarrow}$ in $\xi$, where $Q$ is quasi-$\xi$-projective and $\begin{tiny}\begin{bmatrix}
		b \\  c\end{bmatrix}\end{tiny}$ is essential. By $\rm(ET4)^{op}$, we get the following commutative diagram made of $\mathbb{E}$-triangles
	\[
		\xymatrix{
			Y\ar@{=}[d] \ar[r]^g & P'  \ar[r]^{f}\ar[d]^{d} & A \ar[d]^{\begin{tiny}\begin{bmatrix}
				1 \\  0\end{bmatrix}\end{tiny}} \ar@{-->}[r]^{\kappa=\theta^*\begin{tiny}\begin{bmatrix}
					1 \\  0\end{bmatrix}\end{tiny}} &  \\
			Y\ar[r]^a & Q\ar[d]^{e} \ar[r]^{\begin{tiny}\begin{bmatrix}
				b \\  c\end{bmatrix}\end{tiny}\quad} & A\oplus P\ar[d]^{\begin{tiny}\begin{bmatrix}
				0 & 1\end{bmatrix}\end{tiny}} \ar@{-->}[r]^{\theta} &  \\
			& P\ar@{=}[r]\ar@{-->}[d]^{\eta=0}&P\ar@{-->}[d]^{0} \\
			& & &  } 
	\] 
with $\eta=g_*0=0,\kappa=\theta^*\begin{tiny}\begin{bmatrix}
	1 \\  0\end{bmatrix}\end{tiny}$. Hence  $P'\stackrel{d}\longrightarrow Q \stackrel{e}\longrightarrow P\stackrel{\eta}{\dashrightarrow}$ is a split $\mathbb{E}$-triangle in $\xi$. So there is a morphism $e'\in\mathcal{C}(P,Q)$ such that $ee'=1$ by  Lemma \ref{FJ}. And the $\mathbb{E}$-triangle $Y\stackrel{g}{\longrightarrow} P' \stackrel{f}{\longrightarrow} A\stackrel{\kappa}{\dashrightarrow}$ in $\xi$, because $\xi$ is closed under base change. 

Now we claim that the morphism $f$ is a essential $\xi$-deflation. Assume that there is a morphism $h\in\mathcal{C}(D,P')$ such that the  composite morphism $fh:D\to A$ is a $\xi$-deflation. Recall the above commutative diagram. We have  $e=\begin{tiny}\begin{bmatrix}
	0 &  1\end{bmatrix}\end{tiny}\begin{tiny}\begin{bmatrix}
	b \\  c\end{bmatrix}\end{tiny}=c$ and $\begin{tiny}\begin{bmatrix}
		b \\  c\end{bmatrix}\end{tiny}d=\begin{tiny}\begin{bmatrix}
	1 \\  0\end{bmatrix}\end{tiny}f$ which implies that $e=c,bd=f$ and $cd=0$. Note that 
	 \[
		\begin{tiny}\begin{bmatrix}
			b \\  c\end{bmatrix}\end{tiny}\begin{tiny}\begin{bmatrix}
				dh &  e'\end{bmatrix}\end{tiny}=\begin{tiny}\begin{bmatrix}
					bdh &  be'\\cdh&ce'\end{bmatrix}\end{tiny}=\begin{tiny}\begin{bmatrix}
						fh &  be'\\0&1\end{bmatrix}\end{tiny}\in\mathcal{C}(D\oplus P,A\oplus P)
	\]
since the morphism $\begin{tiny}\begin{bmatrix} dh &  e'\end{bmatrix}\end{tiny}$ lies in $\mathcal{C}(D\oplus P,Q)$. Consider the following commutative diagram: 
\[
		\xymatrix{
			D \ar[d]_{fh} \ar[r]^{\begin{tiny}\begin{bmatrix}
				1\\0
				\end{bmatrix}\end{tiny}\quad} & D\oplus P  \ar[r]^{\quad \begin{tiny}\begin{bmatrix}
					0&1
					\end{bmatrix}\end{tiny}}\ar[d]^{\begin{tiny}\begin{bmatrix}
						fh & be'\\0 & 1
						\end{bmatrix}\end{tiny}} & P \ar@{=}[d] \ar@{-->}[r]^{0} &  \\
			A\ar[r]_{\begin{tiny}\begin{bmatrix}
				1\\0
				\end{bmatrix}\end{tiny}\quad} & A\oplus P \ar[r]_{\quad\begin{tiny}\begin{bmatrix}
					0&1
					\end{bmatrix}\end{tiny}} & P \ar@{-->}[r]^{0} & }
	\] 
Then  the morphism $\begin{tiny}\begin{bmatrix} 
	fh & be'\\0 & 1
\end{bmatrix}\end{tiny}$ is a $\xi$-deflation by Lemma \ref{MDe}. Since the morphism $\begin{tiny}\begin{bmatrix}
	b \\  c\end{bmatrix}\end{tiny}$ is a essential  $\xi$-deflation, the morphism $\begin{tiny}\begin{bmatrix}
		dh &  e'\end{bmatrix}\end{tiny}$ is also a essential  $\xi$-deflation.  Therefore, there exists an  $\mathbb{E}$-triangle  $K\longrightarrow D\oplus P \stackrel{\begin{tiny}\begin{bmatrix}
			dh &  e'\end{bmatrix}\end{tiny}}\longrightarrow Q \dashrightarrow$ in $\xi$. Consider the following commutative diagram: 
		\[
				\xymatrix{
					D \ar[d]_{h} \ar[r]^{\begin{tiny}\begin{bmatrix}
						1\\0
						\end{bmatrix}\end{tiny}\quad} & D\oplus P  \ar[r]^{\quad \begin{tiny}\begin{bmatrix}
							0&1
							\end{bmatrix}\end{tiny}}\ar[d]^{\begin{tiny}\begin{bmatrix}
								dh &  e'\end{bmatrix}\end{tiny}} & P \ar@{=}[d] \ar@{-->}[r]^{0} &  \\
					P'\ar[r]^{d} & Q \ar[r]^{e} & P \ar@{-->}[r]^{0} & }
			\] 
So the morphism $h$ is deflation by Proposition \ref{FDe}. Thus there is an  $\mathbb{E}$-triangle  $K'\stackrel{k}\longrightarrow D \stackrel{h}\longrightarrow P' \dashrightarrow$ in $\mathcal{C}$. By \cite[Lemma 4.14]{JDP}, there exists following commutative diagram made of $\mathbb{E}$-triangles.
\[
	\xymatrix{
		K'\ar[d]_{k}\ar[r]&K\ar[d]\ar[r] & 0\ar[d] \ar@{-->}[r]^{0}&\\
		D \ar[d]_{h} \ar[r]^{\begin{tiny}\begin{bmatrix}
			1\\0
			\end{bmatrix}\end{tiny}\quad} & D\oplus P  \ar[r]^{\quad \begin{tiny}\begin{bmatrix}
				0&1
				\end{bmatrix}\end{tiny}}\ar[d]^{\begin{tiny}\begin{bmatrix}
					dh &  e'\end{bmatrix}\end{tiny}} & P \ar@{=}[d] \ar@{-->}[r]^{0} && (*) \\
		P'\ar[r]^{d} \ar@{-->}[d]& Q\ar@{-->}[d] \ar[r]^{e} & P \ar@{-->}[d]\ar@{-->}[r]^{0} &\\
		&& &}
\] 
We fix $P_0\in\P$. Applying $\mathcal{C}(P_0,-)$ to the diagram $(*)$, one obtains the following commutative diagram of abelian groups:
\[
	\xymatrix{
		&0\ar@{-->}[d]&0\ar[d]&0\ar[d]\\
		0\ar[r]&\mathcal{C}(P_0,K')\ar[d]_{\mathcal{C}(P_0,k)}\ar[r]&\mathcal{C}(P_0,K)\ar[d]\ar[r] & 0\ar[d]\ar[r]&0\\
		0\ar[r]&\mathcal{C}(P_0,D) \ar[d]_{\mathcal{C}(P_0,h)}\ar[r] & \mathcal{C}(P_0,D\oplus P)  \ar[r]\ar[d]& \mathcal{C}(P_0,P) \ar@{=}[d] \ar[r]&0&(**)  \\
		0\ar[r]&\mathcal{C}(P_0,P')\ar@{-->}[d]\ar[r] & \mathcal{C}(P_0,Q)\ar[d] \ar[r] & \mathcal{C}(P_0,P) \ar[d] \ar[r]&0\\
		&0&0&0
			}
\]
Since the split $\mathbb{E}$-triangles belong to $\xi$, all the rows  and the right columns of $(**)$ are exact.
Note that the $\mathbb{E}$-triangle $K\longrightarrow D\oplus P \stackrel{\begin{tiny}\begin{bmatrix}
			dh &  e'\end{bmatrix}\end{tiny}}\longrightarrow Q \dashrightarrow$ lies in $\xi$,
so the middle column of $(**)$ is exact. It is easy to see that the morphism $\mathcal{C}(P_0,k)$ is monic. Now Applying the Snake Lemma to the bottow two rows in $(**)$, it follows that the morphism $\mathcal{C}(P_0,h)$ is epic.  By definition we have shown that $K'\stackrel{k}\longrightarrow D \stackrel{h}\longrightarrow P' \dashrightarrow$ is a $\Mcc(\P,-)$-exact $\mathbb{E}$-triangle. Thus $K'\stackrel{k}\longrightarrow D \stackrel{h}\longrightarrow P' \dashrightarrow$ lies in $\xi$ by Lemma \ref{FBZH1}. So the morphism $h$ is a $\xi$-deflation, which implies that  $f$ is a essential $\xi$-deflation.

Next we claim that $P'$ is $\xi$-projective. Since $P$ is $\xi$-projective, we obtain a short exact sequence
\[
	0\longrightarrow\Mcc(P,Y)\longrightarrow\Mcc(P,P')\stackrel{\mathcal{C}(P,f)}\longrightarrow\Mcc(P,A)\longrightarrow 0.
\]
So there is a morphism $t:P\to P'$ such that $y=ft$. Note that $y$ is a $\xi$-deflation and  $f$ is essential. Therefore, the morphism $t$  is also a $\xi$-deflation. Recall the  $\mathbb{E}$-triangle $P'\stackrel{d}\longrightarrow Q \stackrel{e}\longrightarrow P\stackrel{\eta}{\dashrightarrow}$ is  split, which implies $P'\oplus P\simeq Q$. Thus $P'\oplus P$ is a quasi-$\xi$-projective object since $\QP$ is closed under the isomorphism. Therefore, $P'$ lies in $\P$ by Lemma \ref{TS}. Now it is enough to say $f:P'\to A$ is a $\xi$-projective cover of $A$. 
\end{proof}

Dually, there is following result.

\begin{thm}\label{HZG}
	Every object in $\mathcal{C}$ possesses a $\xi$-injective envelope if and only if  every object in $\mathcal{C}$ possesses a quasi-$\xi$-injective envelope.
\end{thm}

\section{Quasi-$\xi$-$\mathcal{G}$projective objects and  some fundamental properties }

\begin{defn}\citep[Definition 4.4]{JDP}
	An unbounded complex $\mathbf{X}$ is called $\xi$-exact if   $\mathbf{X}$ is a diagram
	\[
	\xymatrix{
		\cdots\ar[r] &X_1\ar[r]^{d_1}&X_0 \ar[r]^{d_0}&X_{-1}\ar[r]&\cdots
	}
	\]
	in $\mathcal{C}$ such that for each integer $n$, there exists an $\mathbb{E}$-triangle $K_{n+1}\stackrel{g_n}\longrightarrow X_n\stackrel{f_n}\longrightarrow K_n\stackrel{\delta_n}\dashrightarrow$ in $\xi$ and $d_n=g_{n-1}f_n$. These $\mathbb{E}$-triangles are called the $\xi$-resolution $\mathbb{E}$-triangles of the $\xi$-exact complex $\mathbf{X}$.
\end{defn}

	A complex $\mathbf{X}$ is called $\Mcc(-,\QP)$-exact   if it is a $\xi$-exact complex with $\Mcc(-,\QP)$-exact resolution $\mathbb{E}$-triangles.

\begin{defn}\label{X1}
	Let $	\mathbf{P}$ is a $\Mcc(-,\QP)$-exact  complex in $\Mcc$ as follows
	\[
	\mathbf{P}: \xymatrix{ \cdots\ar[r]&P_1\ar[r]^{d_1}&P_{0}\ar[r]^{d_0}&P_{-1}\ar[r]&{\cdots}}
	\]
	where $P_n$ is $\xi$-projective for each integer $n$. In this case, for any integer $n$, there exists a $\Mcc(-,\QP)$-exact $\mathbb{E}$-triangle  $K_{n+1}\stackrel{g_n}\longrightarrow X_n\stackrel{f_n}\longrightarrow K_n\stackrel{\delta_n}\dashrightarrow$ in $\xi$ which is the $\xi$-resolution $\mathbb{E}$-triangle of $\mathbf{P}$. Then the objects $K_n$ are called quasi-$\xi$-Gorenstein projective (quasi-$\xi$-$\mathcal{G}$projective for short) for each integer $n$. Dually we can define the quasi-$\xi$-Gorenstein injective  (quasi-$\xi$-$\mathcal{G}$injective for short) objects. We denote by $\mathcal{Q}\Mcg\Mcp(\xi)$ (rsep. $\mathcal{Q}\GI$) the full subcategory of  quasi-$\xi$-$\Mcg$projective (quasi-$\xi$-$\Mcg$injective) objects in $\Mcc$.
\end{defn}

\begin{rem}
	(1) Any $\xi$-projective object is  quasi-$\xi$-$\mathcal{G}$projective.

	(2) Any  quasi-$\xi$-$\mathcal{G}$projective object is $\xi$-$\mathcal{G}$projective (see \citep[Definition 4.7 and 4.8]{JDP}).
\end{rem}

\begin{proof}
	(1) For any $P\in\P$ the $\mathbb{E}$-triangle $P = P\to 0\dashrightarrow$ lies in $\xi$ since $\Delta_0 \subseteq \xi$. For any $Q\in\QP$, applying $\Mcc(-,Q)$ to the given $\mathbb{E}$-triangle $P = P\to 0\dashrightarrow$ gives the exact sequence $0\to \Mcc(P,Q) = \Mcc(P,Q)\to 0 \to 0$, and so this $\mathbb{E}$-triangle is also $\Mcc(-,\QP)$-exact. Hence the complex ${\cdots}\to 0\to P=P\to 0\to {\cdots}$ is a $\QP$-exact complex, and so $P\in \QGP$ by again considering the $\mathbb{E}$-triangle $P = P\to 0\dashrightarrow$. 

	(2) It is obvious since $\P\in\QP$.
\end{proof}
   
\begin{lem}\label{HHH}
Let $A\stackrel{x}\longrightarrow B \stackrel{y}\longrightarrow C  \stackrel{\delta}\dashrightarrow$  be a $\mathbb{E}$-triangle in $\xi$.

(1) If $C$ is quasi-$\xi$-$\mathcal{G}$projective, then the $\mathbb{E}$-triangle $A\stackrel{x}\longrightarrow B \stackrel{y}\longrightarrow C  \stackrel{\delta}\dashrightarrow$ is $\Mcc(-,\QP)$-exact.

(2) If $C$ is a direct summand of quasi-$\xi$-$\mathcal{G}$projective, then the $\mathbb{E}$-triangle $A\stackrel{x}\longrightarrow B \stackrel{y}\longrightarrow C  \stackrel{\delta}\dashrightarrow$ is $\Mcc(-,\QP)$-exact.
\end{lem}

\begin{proof}
	(1) Since $A\stackrel{x}\longrightarrow B \stackrel{y}\longrightarrow C\stackrel{\delta}\dashrightarrow$  is a $\mathbb{E}$-triangle in $\xi$ with $C\in \QGP$, there is a $\Mcc(-,\QP)$-exact $\mathbb{E}$-triangle  $K \stackrel{f}\longrightarrow P \stackrel{g}\longrightarrow C  \stackrel{\theta}\dashrightarrow$ in $\xi$ with $P\in\P$ and $K\in\QGP$. Hence there exists a commutative diagram 
\[
\xymatrix{
	&K\ar@{=}[r]\ar[d]^{f_1}&K\ar[d]^f\\
	A\ar[r]^{x_1}\ar@{=}[d]&M\ar[r]^{y_1}\ar[d]^{g_1}&P\ar[d]^g\ar@{-->}[r]^{g^*\delta}&\\
	A\ar[r]^x&B\ar[r]^y\ar@{-->}[d]^{y^*\theta}&C\ar@{-->}[r]^{\delta}\ar@{-->}[d]^{\theta}&\\
	&&&
}
\]	
made of $\mathbb{E}$-triangles by Lemma \ref{BH} (1). Note that $P$ is $\xi$-projective, so $g$ factors through $y$. Then the the $\mathbb{E}$-triangle $A\stackrel{x_1}\longrightarrow M\stackrel{y_1}\longrightarrow P \stackrel{g^*\delta}\dashrightarrow $ is split by  Lemma \ref{FJ}, hence it is a $\Mcc(-,\QP)$-exact $\mathbb{E}$-triangle in $\xi$. Therefore, the $\mathbb{E}$-triangle 
$A\stackrel{x}\longrightarrow B \stackrel{y}\longrightarrow C  \stackrel{\delta}\dashrightarrow$
 is $\Mcc(-,\QP)$-exact by \citep[Lemma 4.10 (1)]{JDP}.

 (2) Suppose that $C\oplus C'$ is a quasi-$\xi$-$\mathcal{G}$projective object, then we have the following commutative diagram 
 \[
 \xymatrix{
	 &C'\ar@{=}[r]\ar[d]^{f_1}&C'\ar[d]^{\begin{tiny}\begin{bmatrix}
		0\\1
		\end{bmatrix}\end{tiny}}\\
	 A\ar[r]^{x_1}\ar@{=}[d]&M\ar[r]^{y_1}\ar[d]^{g_1}&C\oplus C'\ar[d]^{\begin{tiny}\begin{bmatrix}
		1&0
		\end{bmatrix}\end{tiny}}\ar@{-->}[r]^{\qquad\begin{tiny}\begin{bmatrix}
			1&0
			\end{bmatrix}\end{tiny}^*\delta}&\\
	 A\ar[r]^x&B\ar[r]^y\ar@{-->}[d]^{y^*\theta}&C\ar@{-->}[r]^{\delta}\ar@{-->}[d]^{\theta}&\\
	 &&&
 }
 \]	
 made of $\mathbb{E}$-triangles in $\xi$ by Lemma \ref{BH} (1). Note that the second horizontal is $\Mcc(-,\QP)$-exact by (1) and the third vertical  is  $\Mcc(-,\QP)$-exact, then so is the third horizontal by \citep[Lemma 4.10 (1)]{JDP}.
\end{proof}

In the following part, we give some characterizations of quasi-$\xi$-$\mathcal{G}$projective objects.

\begin{lem}\label{TTTT}
	Assume that $G$ is an object  in $\QGP$. If $Q\in\QP$, then $\ext^0(G,Q) \simeq \Mcc(G,Q)$ and $\ext^n(G,Q)=0$ for any  $n\geq 1$.
\end{lem}
\begin{proof}
	Let $G'\longrightarrow P \longrightarrow G  \dashrightarrow$ be an $\mathbb{E}$-triangle in $\xi$ with $G'\in\QGP$ and $P\in\P$. If $Q\in\QP$, then we have the following commutative diagram 
\[
	\xymatrix{
		0\ar[r]&\Mcc(G,Q)\ar[r]\ar[d]_{\varphi_1}&\Mcc(P,Q)\ar[r]\ar[d]^{\varphi_2}_{\simeq }&\Mcc(G',Q)\ar[r]\ar[d]^{\varphi_3}& 0&\\
	0\ar[r]&\ext^0(G,Q)\ar[r]&\ext^0(P,Q)\ar[r]&\ext^0(G',Q)\ar[r]&\ext^1(G,Q)\ar[r]&0.
	}
\]
where the top exact sequence follows from Lemma \ref{HHH}(1). Note that  $\varphi_1$ and  $\varphi_3$ are monic, hence $\varphi_1$ is  epic by Snake Lemma, so $\varphi_1$ is an isomorphism. Similarly, one can get that $\varphi_3$ is an isomorphism, so $\ext^1(G,Q)=0$. It is easy to show that $\ext^n(G, M)=0$ for any $n\geq 1$ by Lemma \ref{LZHL}.
\end{proof}

\begin{lem}\label{Dec30}
	(1) If $A\stackrel{f}\longrightarrow B \stackrel{f'}\longrightarrow C  \stackrel{\delta}\dashrightarrow$ and $B\stackrel{g}\longrightarrow D\stackrel{g'}\longrightarrow E  \stackrel{\delta'}\dashrightarrow$ are both $\Mcc(-,\QP)$-exact   $\mathbb{E}$-triangles in $\xi$, then we have the following commutative diagram:
\[
		\xymatrix{
			A\ar@{=}[d] \ar[r]^{f} & B  \ar[r]^{f'}\ar[d]^{g} & C \ar[d]^d \ar@{-->}[r]^{\delta} &  \\
			A\ar[r]^{h} & D\ar[d]^{g'} \ar[r]^{h'} & F\ar[d]^e \ar@{-->}[r]^{\delta''} & \\
			& E\ar@{=}[r]\ar@{-->}[d]^{\delta'}&E\ar@{-->}[d]^{f'_*\delta'} \\
			& & &  }
	\] 
where all rows and columns are both $\Mcc(-,\QP)$-exact  $\mathbb{E}$-triangles in $\xi$.

(2) If $A\stackrel{f}\longrightarrow B \stackrel{f'}\longrightarrow C  \stackrel{\delta}\dashrightarrow$ and $D\stackrel{g}\longrightarrow C\stackrel{g'}\longrightarrow E  \stackrel{\delta'}\dashrightarrow$ are both $\Mcc(-,\QP)$-exact   $\mathbb{E}$-triangles in $\xi$, then we have the following commutative diagram:
\[
		\xymatrix{
			A\ar@{=}[d] \ar[r]^{d} & F  \ar[r]^{e}\ar[d]^{h} & D \ar[d]^{g} \ar@{-->}[r]^{g^*\delta} &  \\
			A\ar[r]^f & B\ar[d]^{h'} \ar[r]^{f'} & C\ar[d]^{g'} \ar@{-->}[r]^{\delta} & \\
			& E\ar@{=}[r]\ar@{-->}[d]^{\delta''}&E\ar@{-->}[d]^{\delta'} \\
			& & &  }
	\] 
where all rows and columns are both $\Mcc(-,\QP)$-exact  $\mathbb{E}$-triangles in $\xi$.
\end{lem}
\begin{proof}
	(1) It follows from \citep[Theorem 3.2]{JDP} and $\rm(ET4)$ that we have the desired commutative diagram where all  rows and columns are both  $\mathbb{E}$-triangles in $\xi$. For any object $Q\in\QP$, $\Mcc(h,Q)$ is an epimorphism since $ \Mcc(h,Q)=\Mcc(f,Q)\Mcc(g,Q)$. And it is easy to check that $\Mcc(h,Q)$ is a monomorphism by \citep[Lemma 3(2)]{JDPPR}. This implies that the  $\mathbb{E}$-triangle $A\stackrel{h}\longrightarrow D \stackrel{h'}\longrightarrow F \stackrel{\delta''}\dashrightarrow$ is  $\Mcc(-,\QP)$-exact. It is easy to get that the  $\mathbb{E}$-triangle $C\stackrel{d}\longrightarrow F \stackrel{e}\longrightarrow E \stackrel{\delta''}\dashrightarrow$  is  $\Mcc(-,\QP)$-exact by $3\times 3$-Lemma.

	(2) It is similar to the proof of (1).
\end{proof}

 A class of $\mathcal{X}$ is called $\xi$-projectively resolving if $\P\subseteqq \mathcal{X}$, and for any  $\mathbb{E}$-triangle $A\longrightarrow B \longrightarrow C \dashrightarrow$ in $\xi$ with $C\in\mathcal{X}$, the conditions $A\in\mathcal{X}$ and $B\in\mathcal{X}$ are equivalent. Also, a class of $\mathcal{Y}$ is called $\xi$-injectively resolving if $\I\subseteqq \mathcal{Y}$, and for any  $\mathbb{E}$-triangle $A\longrightarrow B \longrightarrow C \dashrightarrow$ in $\xi$ with $A\in\mathcal{Y}$, the conditions $B\in\mathcal{Y}$ and $C\in\mathcal{Y}$ are equivalent.

\begin{thm}\label{Dec31}
	$\QGP$ is $\xi$-projectively resolving.
\end{thm}

\begin{proof}
	Let $A\stackrel{x}\longrightarrow B \stackrel{y}\longrightarrow C  \stackrel{\delta}\dashrightarrow$ be an $\mathbb{E}$-triangle in $\xi$ with $C\in\QGP$, then the  $\mathbb{E}$-triangle 	 $A\stackrel{x}\longrightarrow B \stackrel{y}\longrightarrow C  \stackrel{\delta}\dashrightarrow$ is   $\Mcc(-,\QP)$-exact by Lemma \ref{HHH}(1).
	
	If $A$ is in $\QGP$, then there are two $\Mcc(-,\QP)$-exact $\mathbb{E}$-triangles
	\[
		A\stackrel{g^A_{-1}}\longrightarrow P^A_{-1} \stackrel{f^A_{-1}}\longrightarrow K^A_0  \stackrel{\delta^A_{-1}}\dashrightarrow \text{\quad and\quad } C\stackrel{g^C_{-1}}\longrightarrow P^C_{-1} \stackrel{f^C_{-1}}\longrightarrow K^C_0  \stackrel{\delta^C_{-1}}\dashrightarrow
	\]
	in $\xi$ with $P^A_{-1},P^C_{-1}\in\P$ and $K^A_0,K^C_0\in\QGP$. By \citep[Lemma 5]{JDPPR} there is the following commutative diagram

	\[
		\xymatrix{
		A\ar[r]^{x}\ar[d]_{g_{-1}^A}&B\ar[r]^{y}\ar[d]^{g_{-1}^B}&C\ar@{-->}[r]^{\delta}\ar[d]^{g_{-1}^C}&\\
		P_{-1}^A\ar[r]^{\begin{tiny}\begin{bmatrix}
		1 \\
		0
		\end{bmatrix}\end{tiny}}\ar[d]_{f_{-1}^A}&P^B_{-1}\ar[r]^{\begin{tiny}\begin{bmatrix}
		0&1
		\end{bmatrix}\end{tiny}}\ar[d]^{f_{-1}^B}&P_{-1}^C\ar@{-->}^0[r]\ar[d]^{f_{-1}^C}&\\
		K^A_0\ar[r]^x\ar@{-->}[d]^{\delta_{-1}^A}& K^B_0\ar[r]^y\ar@{-->}[d]^{\delta_{-1}^B}& K^C_0\ar@{-->}[r]^{\delta}\ar@{-->}[d]^{\delta_{-1}^C}&\\
		&&&
		} 
	\]
where all rows and columns are $\Mcc(-,\QP)$-exact $\mathbb{E}$-triangles in $\xi$ with $P^B_{-1}=:P_{-1}^A\oplus P_{-1}^C$. Since $K^A_0,K^C_0$ belong to $\QGP$, by repeating this process, we can obtain a $\Mcc(-,\QP)$-exact complex 
\[
	\xymatrix{
		 B\ar[r]&P^B_{-1}\ar[r]&P^B_{-2}\ar[r]&P^B_{-3}\ar[r]&\cdots
	}
\]
Similarly, we can obtain a $\Mcc(-,\QP)$-exact complex 
\[
	\xymatrix{
		\cdots\ar[r]&P^B_2\ar[r]&P^B_1\ar[r]&P^B_0\ar[r]&B.
	}
\]
By pasting these  $\Mcc(-,\QP)$-exact complexes together, we obtain the follows $\Mcc(-,\QP)$-exact complex 
\[
	\xymatrix{
		\cdots\ar[r]&P^B_2\ar[r]&P^B_1\ar[r]&P^B_0\ar[r]&P^B_{-1}\ar[r]&P^B_{-2}\ar[r]&P^B_{-3}\ar[r]&\cdots
	}
\]
which implies $B$ is in $\QGP$.

If $B$  is in $\QGP$,  there is a $\Mcc(-, \QP)$-exact $\mathbb{E}$-triangle
\[
	B\stackrel{g^B_{-1}}\longrightarrow P^B_{-1} \stackrel{f^B_{-1}}\longrightarrow K^B_0  \stackrel{\delta^B_{-1}}\dashrightarrow
\]
in $\xi$ with $P_{-1}^B \in \P$ and $K_{0}^B \in \QGP$. By Lemma \ref{Dec30}(1), there exists the following commutative diagram
\[
	\xymatrix{
		A\ar@{=}[d] \ar[r]^x & B  \ar[r]^{y}\ar[d]^{g^B_{-1}} & C \ar[d]^{g} \ar@{-->}[r]^{\delta} &  \\
		A\ar[r]^{g^A_{-1}} & P^B_{-1}\ar[d]^{f^B_{-1}} \ar[r]^{f^A_{-1}} & G\ar[d]^{f} \ar@{-->}[r]^{\delta^A_{-1}} &  \\
		& K^B_0\ar@{=}[r]\ar@{-->}[d]^{\delta^B_{-1}}& K^B_0\ar@{-->}[d]^{y_*\delta^B_{-1}} \\
		& & &  } 
\] 
where all rows and columns are both $\Mcc(-,\QP)$-exact  $\mathbb{E}$-triangles in $\xi$.  Since $G$ lies in $\QGP$, there is a $\Mcc(-,\QP)$-exact complex
\[
	\xymatrix{
		 G\ar[r]&P^A_{-2}\ar[r]&P^A_{-3}\ar[r]&P^A_{-4}\ar[r]&\cdots
	}
\]
with $P^A_n\in \P$ for any $n\geq 2$.  Hence we get a $\Mcc(-,\QP)$-exact $\xi$-exact complex
\[\xymatrix{
	A\ar[r]& P^B_{-1}\ar[r]&P^A_{-2}\ar[r]&P^A_{-3}\ar[r]&\cdots
	}
\]
with $P_{-1}^B \in \P$ and $P_{-n}^A \in \P$ for any $n \geqslant 2$. 

Since $\Mcc$ has enough $\xi$-projective and $C\in\QGP$, there exist two  $\mathbb{E}$-triangles
\[
	K_1^A\stackrel{g^A_0}\longrightarrow P^A_0 \stackrel{f^A_0}\longrightarrow A  \stackrel{\delta^A_0}\dashrightarrow \quad\text{and}\quad K_1^C\stackrel{g^C_0}\longrightarrow P^C_0 \stackrel{f^C_0}\longrightarrow C  \stackrel{\delta^C_0}\dashrightarrow
\]
in $\xi$ with  $P_0^A, P_0^C \in \P$ and $K_1^C \in\QGP$. Because of the  $\mathbb{E}$-triangle $A\stackrel{x}\longrightarrow B \stackrel{y}\longrightarrow C  \stackrel{\delta}\dashrightarrow$ is   $\Mcc(\P,-)$-exact, there  exists a morphism $a\in\Mcc(P^C_0,B)$ such that $f_0^C=ya$. So $(f_0^C)^*\delta=0=(f_0^A)_*$ by Lemma \ref{FJ} and there is a $\xi$-deflation $f^B_0:P_0^A\oplus P_0^C=:P_0^B\to B $ by \citep[Proposition 1]{JDPPR}, which makes the following diagram 
\[
	\xymatrix{P_0^A\ar[r]^{\begin{tiny}\begin{bmatrix} 
	1 \\
	0
	\end{bmatrix}\end{tiny}}\ar[d]_{f^A_0}&P_0^B\ar[r]^{\begin{tiny}\begin{bmatrix}
	0&1
	\end{bmatrix}\end{tiny}}\ar@{-->}[d]^{f^B_0}&P_0^C\ar@{-->}^0[r]\ar[d]^{f^C_0}&\\
	A\ar[r]^x&B\ar[r]^y&C\ar@{-->}[r]^{\delta}&
	}
	\]
commutative. Assume that $K_1^B\stackrel{g^B_0}\longrightarrow P^B_0 \stackrel{f^b_0}\longrightarrow B  \stackrel{\delta^B_0}\dashrightarrow$ is an $\mathbb{E}$-triangle in $\xi$ which is $\Mcc(-,\QP)$-exact. Then we have the following commutative diagram 
\[
	\xymatrix{
	K_1^A\ar[r]^{x_1}\ar[d]_{g^A_0}&K_1^B\ar[r]^{y_1}\ar[d]^{g_0^B}&K_1^C\ar@{-->}[r]^{\delta_1}\ar[d]^{g_0^C}&\\
	P_0^A\ar[r]^{\begin{tiny}\begin{bmatrix}
	1 \\
	0
	\end{bmatrix}\end{tiny}}\ar[d]_{f_0^A}&P_0^B\ar[r]^{\begin{tiny}\begin{bmatrix}
	0&1
	\end{bmatrix}\end{tiny}}\ar[d]^{f_0^B}&P_0^C\ar@{-->}^0[r]\ar[d]^{f_0^C}&\\
	A\ar[r]^x\ar@{-->}[d]^{\delta_0^A}&B\ar[r]^y\ar@{-->}[d]^{\delta_0^B}&C\ar@{-->}[r]^{\delta}\ar@{-->}[d]^{\delta_0^C}&\\
	&&&
	} 
\]
where all rows and columns are $\mathbb{E}$-triangles in $\xi$ by \citep[Lemma 4.14]{JDP}. It is easy to show that the first vertical is $\Mcc(-,\QP)$ by $3\times 3$-Lemma. Recall that  $B$  is in $\QGP$, so there is an $\mathbb{E}$-triangles 
$K'^B_0\stackrel{g'^B_{0}}\longrightarrow P'^B_{0} \stackrel{f'^B_{0}}\longrightarrow B\stackrel{\delta'^B_{0}}\dashrightarrow $
in $\xi$ by definition, where $K'^B_0\in\QGP$ and $P'^B_{0}\in\P$. Then we have $K^B_0\oplus P'^B_{0}\backsimeq K'^B_0\oplus P^B_{0}\in\QGP$. Hence, any $\mathbb{E}$-triangle $K_2^B\stackrel{g^B_1}\longrightarrow P^B_1 \stackrel{f^B_1}\longrightarrow K^B_1  \stackrel{\delta^B_1}\dashrightarrow$ in $\xi$ is $\Mcc(-,\QP)$-exact by Lemma \ref{HHH} (2). Repeating this process, we can obtain a $\Mcc(-,\QP)$-exact complex 
\[
	\xymatrix{
		\cdots\ar[r]&P^A_2\ar[r]&P^A_1\ar[r]&P^A_0\ar[r]&A
	}
\]
with $P^A_n\in\P$ for any $n\geq 0$. Therefore, we obtain that  $A$ is in $\QGP$, as desired.
\end{proof}

\begin{thm}\label{Jan1}
	$\QGP$ is closed under direct summands.
\end{thm}

\begin{proof}
	Assume that $G\in \QGP$ and $M$ is a direct summand of $G$. Then there exists $M'\in\Mcc$ such that $G=M\oplus M'$. Therefore, there exist two split $\mathbb{E}$-triangles
\[
M\stackrel{\begin{tiny}\begin{bmatrix}
	1 \\
	0
	\end{bmatrix}\end{tiny}}\longrightarrow G \stackrel{\begin{tiny}\begin{bmatrix}
		0&1
		\end{bmatrix}\end{tiny}}\longrightarrow M'\stackrel{0}{\dashrightarrow} \text{\ and\ }
		M'\stackrel{\begin{tiny}\begin{bmatrix}
			0 \\
			1
			\end{bmatrix}\end{tiny}}\longrightarrow G \stackrel{\begin{tiny}\begin{bmatrix}
				1&0
				\end{bmatrix}\end{tiny}}\longrightarrow M\stackrel{0}{\dashrightarrow} 
\]
in $\xi$. Since $G\in\QGP$, there exists a  $\Mcc(-,\QP)$-exact $\mathbb{E}$-triangle $G\stackrel{g_{-1}}\longrightarrow P_{-1} \stackrel{f_{-1}}\longrightarrow K_{0} \stackrel{\delta_{-1}}\dashrightarrow$ in $\xi$ with $P_{-1}\in\P$ and $K_0\in\QGP$. By Lemma \ref{Dec30}(1), there exists the following commutative diagram
\[
		\xymatrix{
			M\ar@{=}[d] \ar[r]^{\begin{tiny}\begin{bmatrix}
				1 \\
				0
				\end{bmatrix}\end{tiny}} & G  \ar[r]^{\begin{tiny}\begin{bmatrix}
					0&1
					\end{bmatrix}\end{tiny}}\ar[d]^{g_{-1}} & M' \ar[d]^{x'_{-1}} \ar@{-->}[r]^{0} &  \\
			M\ar[r]^{x_{-1}} & P_{-1}\ar[d]^{f_{-1}} \ar[r]^{y_{-1}} & X\ar[d]^{y'_{-1}} \ar@{-->}[r]^{\theta_{-1}} & \\
			& K_0\ar@{=}[r]\ar@{-->}[d]^{\delta_{-1}}&K_0\ar@{-->}[d]^{\theta'_{-1}=\begin{tiny}\begin{bmatrix}
				0&1
				\end{bmatrix}\end{tiny}_*\delta_{-1}} \\
			& & &  }
	\] 
	where all rows and columns are $\Mcc(-,\QP)$-exact  $\mathbb{E}$-triangles in $\xi$. Note that \[
M'\stackrel{x'_{-1}}\longrightarrow X \stackrel{y'_{-1}}\longrightarrow K_0\stackrel{\theta'_{-1}}{\dashrightarrow} \text{\ and\ }
		M'\stackrel{\begin{tiny}\begin{bmatrix}
			0 \\
			1
			\end{bmatrix}\end{tiny}}\longrightarrow G \stackrel{\begin{tiny}\begin{bmatrix}
				1&0
				\end{bmatrix}\end{tiny}}\longrightarrow M\stackrel{0}{\dashrightarrow} 
\]
are $\mathbb{E}$-triangles in $\xi$.  It follows Lemma \ref{BH}(2) that  we have  the following commutative diagram
 \[
			\xymatrix{
				M'\ar[d]_{\begin{tiny}\begin{bmatrix}
					0 \\
					1
					\end{bmatrix}\end{tiny}}\ar[r]^{x'_{-1}}&X\ar[d]^{g'_{-1}}\ar[r]^{y'_{-1}}&K_0\ar@{=}[d]\ar@{-->}[r]^{\theta'_{-1}} &\\
				G\ar[d]_{\begin{tiny}\begin{bmatrix}
					1&0
					\end{bmatrix}\end{tiny}}\ar[r]^{x''_{-1}}&G_{-1}\ar[r]^{y''_{-1}}\ar[d]^{f'_{-1}}&K_0\ar@{-->}[r]^{\theta''_{-1}} &\\
				M\ar@{=}[r]\ar@{-->}[d]^{0}&M\ar@{-->}[d]^{0}\\
				&&
			}
\]
where \[
G\stackrel{x''_{-1}}\longrightarrow G_{-1} \stackrel{y''_{-1}}\longrightarrow K_0\stackrel{\theta''_{-1}}{\dashrightarrow} \text{\ and\ }
		X\stackrel{g'_{-1}}\longrightarrow G_{-1} \stackrel{f'_{-1}}\longrightarrow M\stackrel{0}{\dashrightarrow} 
\]
are  $\mathbb{E}$-triangles in $\xi$ since $\xi$ is closed under cobase change. Moreover, the  two $\mathbb{E}$-triangles above are $\Mcc(-,\QP)$-exact by Lemma \ref{HHH}. Because $G$ and $K_0$ are in $\QGP$, the object $G_{-1}$ belongs to $\QGP$ by Theorem \ref{Dec31}. Therefore, there is  a  $\Mcc(-,\QP)$-exact $\mathbb{E}$-triangle $G_{-1}\stackrel{g_{-2}}\longrightarrow P_{-2} \stackrel{f_{-2}}\longrightarrow K_{-1} \stackrel{\delta_{-2}}\dashrightarrow$ in $\xi$ with $P_{-2}\in\P$ and $K_{-1}\in\QGP$. So we have the following commutative diagram
\[
		\xymatrix{
			X\ar@{=}[d] \ar[r]^{g'_{-1}} & G_{-1}  \ar[r]^{f'_{-1}}\ar[d]^{g_{-2}} & M \ar[d]^{x'_{-2}} \ar@{-->}[r]^{0} &  \\
			X\ar[r]^{x_{-2}} & P_{-2}\ar[d]^{f_{-2}} \ar[r]^{y_{-2}} & Y\ar[d]^{y'_{-2}} \ar@{-->}[r]^{\theta_{-2}} & \\
			& K_{-1}\ar@{=}[r]\ar@{-->}[d]^{\delta_{-2}}&K_{-1}\ar@{-->}[d]^{(f'_{-1})_*\delta_{-2}} \\
			& & &  }
	\] 
	where all rows and columns are  $\Mcc(-,\QP)$-exact  $\mathbb{E}$-triangles in $\xi$ by Lemma \ref{Dec30}(1). Proceedings this manner, we can obtain a $\Mcc(-,\QP)$-exact complex 
\[\xymatrix{
	M\ar[r]& P_{-1}\ar[r]&P_{-2}\ar[r]&P_{-3}\ar[r]&\cdots
	}
\]
with $P_n\in\P$ for any $n<0$. Similarly, we can get the following  $\Mcc(-,\QP)$-exact complex
\[
	\xymatrix{
		\cdots\ar[r]&P_2\ar[r]&P_1\ar[r]&P_0\ar[r]&M
	}
\]
with $P_n\in\P$ for any $n\geq 0$. Hence, $M\in\QGP$, as desired.
\end{proof}

Dually, there is following results.
\begin{thm}\label{ZEH}
	$\QGI$ is $\xi$-injectively resolving  and closed under direct summands.
\end{thm}  

\begin{cor}
	(1) \citep[Proposition 2.5 and Lemma 2.6]{AF} Assume that $\Mcc$ is a module category and the class $\xi$ is the class of short exact sequences. Then $\QGP$ is $\xi$-projectively resolving  and closed under direct summands.

	(2)  Assume that $\Mcc$ is a triangulated category and the class $\xi$ of triangles is closed under isomorphisms and suspension (see \cite[Section 2.2]{BEL} and \cite[Section 3.3]{HY}). Then $\QGP$ is $\xi$-projectively resolving  and closed under direct summands.
\end{cor}

\begin{prop}\label{BYSJ}
	Let $A\stackrel{x}\longrightarrow B \stackrel{y}\longrightarrow C  \stackrel{\delta}\dashrightarrow$ be an $\mathbb{E}$-triangle in $\xi$ with $A,B\in\QGP$, then $C\in\QGP$ if and only if $\ext^1(C,Q)=0$ for any $Q\in\QP$.
\end{prop}

\begin{proof}
	The ``only if'' part follows from Lemma \ref{TTTT}. For the ``if'' part, since $A\in\QGP$, there is an $\mathbb{E}$-triangle $A\longrightarrow P \longrightarrow K\dashrightarrow$  with $P\in\P$ and $K\in\QGP$. Then we have the following commutative diagram
	\[
		\xymatrix{
			A\ar[d]\ar[r]&B\ar[d]\ar[r]&C\ar@{=}[d]\ar@{-->}[r] &\\
			P\ar[d]\ar[r]&G\ar[r]\ar[d]&C\ar@{-->}[r] &\\
			K\ar@{=}[r]\ar@{-->}[d]&K\ar@{-->}[d]\\
			&&}
\]
where all rows and columns are $\mathbb{E}$-triangles in $\xi$ since $\xi$ is closed under cobase change. It follows from Theorem \ref{Dec31} that $G\in\QGP$. For the $\mathbb{E}$-triangle $P\longrightarrow G \longrightarrow C\dashrightarrow$, we have the following commutative diagram by hypothesis
	\[
	\xymatrix{
	 &\Mcc(C,Q)\ar[r]\ar[d]_{\varphi_1}&\Mcc(G,Q)\ar[r]\ar[d]^{\varphi_2}_{\simeq }&\Mcc(Q,Q)\ar@{-->}[r]\ar[d]_{\simeq }^{\varphi_3}& 0&\\
	0\ar[r]&\ext^0(C,Q)\ar[r]&\ext^0(G,Q)\ar[r]&\ext^0(G',Q)\ar[r]&0.
	}
\]
where $\varphi_2$ and $\varphi_3$ are isomorphisms by Lemma \ref{TTTT}. Hence,the $\mathbb{E}$-triangle $P\longrightarrow G \longrightarrow C\dashrightarrow$ is split and then $C\in\QGP$ by Theorem \ref{Jan1}.
\end{proof}

We set $\Mbe_{\xi}:=\Mbe|_{\xi}$, that is  for any $A$, $C\in\Mcc$, let
		\[
			\Mbe_{\xi}(C,A) = \{\delta\in \Mbe(C,A)~|~\exists A\stackrel{x}\longrightarrow B \stackrel{y}\longrightarrow C  \stackrel{\delta}\dashrightarrow \in\xi\}
		\]
and $\mathfrak{s}_{\xi} := \mathfrak{s}|_{\Mbe_{\xi}}$. By \citep[Theorem 3.2]{JDP}, $(\mathcal{C},\Mbe_{\xi} , \mathfrak{s}_{\xi})$ is also an extriangulated category.
Consider a part of $\Mbe$-triangles in $\xi$  which lie  in the subcategory $\QGP$, and set $\xi_{\QGP}:=\xi|_{\QGP}$, i.e. for an $\Mbe$-triangle $A\longrightarrow B\longrightarrow C  \dashrightarrow$ in $\xi$,  it lies in $\xi_{\QGP}$ if and only if $A, B,C$ are all  belong to $\QGP$.
Set $\Mbe_{\QGP}:=\Mbe_{\xi}|_{\QGP}$, i.e.
		\[
			\Mbe_{\QGP}(C,A) = \{\delta\in \Mbe(C,A)~|~\exists A\stackrel{x}\longrightarrow B \stackrel{y}\longrightarrow C  \stackrel{\delta}\dashrightarrow \in\xi\text{~with~}A,B,C\in\QGP\}
		\]
and $\mathfrak{s}_{\QGP} := \mathfrak{s}_{\xi}|_{\Mbe_{\QGP}}$.
	\begin{lem}
$(\QGP,\Mbe_{\QGP},\mathfrak{s}_{\QGP})$is an extriangulated category.
  \end{lem}

\begin{proof}
Since  $\QGP$ is an extension-closed subcategory of $\Mcc$ by Corollary \ref{Dec31}, it follows directly from \cite[Remark 2.18]{HY}.
\end{proof}

\begin{prop}
	$\xi_{\QGP}$ is a proper class of $(\QGP,\Mbe_{\QGP},\mathfrak{s}_{\QGP})$.
\end{prop}

\begin{proof}
	It is easy to check $\xi_{\QGP}$ satisfies (1) and (2) of Definition \ref{ZL}. Now we show $\xi_{\QGP}$ is saturated.
							
	In the diagram of ~\ref{BH}(1)
		\[
		\xymatrix{
				&{A_2}\ar[d]^{m_2}\ar@{=}[r]&{A_2}\ar[d]^{x_2}\\
				{A_1}\ar@{=}[d]\ar[r]^{m_1}&M\ar[r]^{e_1}\ar[d]^{e_2}&{B_2}\ar[d]^{y_2}\\
				{A_1}\ar[r]^{x_1}&{B_1}\ar[r]^{y_1}&C
				}
		\]
	When $\mathbb{E}$-triangle $A_2\stackrel{x_2}\longrightarrow B_2 \stackrel{y_2}\longrightarrow C  \stackrel{\delta_2}\dashrightarrow$ and $A_1\stackrel{m_1}\longrightarrow M \stackrel{e_1}\longrightarrow B_2  \stackrel{y_2^*\delta_1}\dashrightarrow$  are both	in $\xi_{\QGP}$, we claim that the $\mathbb{E}$-triangle $A_1\stackrel{x_1}\longrightarrow B_1 \stackrel{y_1}\longrightarrow C  \stackrel{\delta_1}\dashrightarrow$ is also in $\xi_{\QGP}$. In fact, regard the diagram as a commutative diagram in $(\mathcal{C},\Mbe_{\xi} , \mathfrak{s}_{\xi})$. Since $A_1\stackrel{x_1}\longrightarrow B_1 \stackrel{y_1}\longrightarrow C  \stackrel{\delta_1}\dashrightarrow$  is in $\xi$ where  $A_1$ and $C$ belong to $\QGP$, so $B_1\in\GP$ by Theorem \ref{Dec31}. Thus $\mathbb{E}$-triangle $A_1\stackrel{x_1}\longrightarrow B_1 \stackrel{y_1}\longrightarrow C  \stackrel{\delta_1}\dashrightarrow$ lies in  $\xi_{\QGP}$. Then we obtain that $\xi_{\QGP}$ is a proper class of the  extriangulated category $(\QGP,\Mbe_{\QGP},\mathfrak{s}_{\QGP})$.
\end{proof}

\begin{prop}\label{TF}
	If $M$ is   quasi-$\xi$-projective  in $(\Mcc,\Mbe,\mathfrak{s})$ , then $M$ is quasi-$\xi_{\QGP}$-injective   in  $(\QGP,\Mbe_{\QGP},\mathfrak{s}_{\QGP})$.
   \end{prop}
   \begin{proof}
   For any $\Mbe$-triangle $A\longrightarrow M \longrightarrow B \dashrightarrow \in\xi_{\QGP}$, there exists long exact sequence
   \[
	   0\longrightarrow\ext^0(B,M)\longrightarrow\ext^0(M,M)\longrightarrow\ext^0(A,M)\longrightarrow\ext^1(B,M).
   \]
   Note that $A, B$ are quasi-$\xi$-$\mathcal{G}$projective  and $M\in\QP$, by Lemma \ref{TTTT}, we have short exact sequence
   \[
	   0\longrightarrow\Mcc(B,M)\longrightarrow\Mcc(M,M)\longrightarrow\Mcc(A,M)\longrightarrow0.
   \]
   So $M$ is quasi-$\xi_{\QGP}$-injective in $(\QGP,\Mbe_{\QGP},\mathfrak{s}_{\QGP})$.
   \end{proof}

\section{Quasi-$\xi$-$\mathcal{G}$projective dimensions and some properties}

\qquad The quasi-$\xi$-$\mathcal{G}$projective dimension $\Qgpd{A}$ of an object $A$ is defined inductively. When $A=0$, put $\Qgpd{A}=-1$.  If $A\in\QGP$, then define $\Qgpd{A}=0$. Next by induction, for an integer $n>0$, put $\Qgpd{A}\leqslant n$ if there exists an $\mathbb{E}$-triangle $K\rightarrow G\rightarrow A\dashrightarrow$ in $\xi$ with $G\in \QGP$ and $\Qgpd{K}\leqslant n-1$. 

We say $\Qgpd{A}=n$ if $\Qgpd{A}\leqslant n$ and $\Qgpd{A} \nleqslant  n-1$. If $\Qgpd{A}\neq n$, for all $n\geqslant0$, we set  $\Qgpd{A}=\infty$. Dually, we can define  the quasi-$\xi$-$\mathcal{G}$injective dimension $\Qgid{A}$ of any object $A$ in $\Mcc$.

We use $\widehat{\mathcal{QGP}}(\xi)$ (resp.  $\widehat{\mathcal{QIP}}(\xi)$)  to denote the full subcategory of $\Mcc$ whose objects have finite  quasi-$\xi$-$\mathcal{G}$projective (resp.  quasi-$\xi$-$\mathcal{G}$injective) dimension.

\begin{lem}\label{YMY}
	Let  $K\longrightarrow G_1\longrightarrow G_0 \longrightarrow A$ be a $\xi$-exact complex with $G_0,G_1\in\QGP$, then there exist two$\xi$-exact complexes  
\[K\longrightarrow P_0\longrightarrow G'_0 \longrightarrow A\text{\quad  and\quad}K\longrightarrow G'_1\longrightarrow P_1 \longrightarrow A\] with $P_0,P_1\in\P$ and  $G'_0,G'_1\in\QGP$.
\end{lem}
\begin{proof}
Since $K\longrightarrow G_1\longrightarrow G_0 \longrightarrow A$ is a $\xi$-exact complex, there exist two $\Mbe$-triangles
\[
		K\longrightarrow G_1\longrightarrow K_1 \dashrightarrow \text{\quad and\quad } K_1\longrightarrow G_0 \longrightarrow A \dashrightarrow
\]
in $\xi$. Since $G_1\in\QGP$, there exists an $\Mbe$-triangle $G_1\longrightarrow P_0\longrightarrow G'_1 \dashrightarrow$
in $\xi$ with $P_0\in\P$ and $G'_1\in\QGP$. By \citep[Theorem 3.2]{JDP} and $\rm(ET4)$, we have the following commutative diagram
\[
		\xymatrix{
			K\ar@{=}[d] \ar[r] & G_1  \ar[r]\ar[d] & K_1 \ar[d] \ar@{-->}[r] &  \\
			K\ar[r]& P_0\ar[d] \ar[r] & B\ar[d] \ar@{-->}[r] & \\
			& G'_1\ar@{=}[r]\ar@{-->}[d]&G'_1\ar@{-->}[d] \\
			& & &  }
	\] 
where all rows and columns are $\Mbe$-triangles in $\xi$. It follows Lemma \ref{BH}(2) that  we have  the following commutative diagram
 \[
			\xymatrix{
				K_1\ar[d]\ar[r]&G_0\ar[d]\ar[r]&A\ar@{=}[d]\ar@{-->}[r]&\\
				B\ar[d]\ar[r]&G'_0\ar[r]\ar[d]&A\ar@{-->}[r]&\\
				G'_1\ar@{=}[r]\ar@{-->}[d]&G'_1\ar@{-->}[d]\\
				&&
			}
\]
where all rows and columns are $\Mbe$-triangles in $\xi$ since $\xi$ is closed under cobase change.  Since $G_0,G_1\in\QGP$, by
Theorem \ref{Dec31}, $G'_0\in\QGP$. Then we get the $\xi$-exact complex $K\longrightarrow P_0\longrightarrow G'_0 \longrightarrow A$ with $P_0\in\P$ and  $G'_0\in\QGP$. Similarly, we can obtain the other required $\xi$-exact complex
\end{proof}
\begin{cor}\label{ZXR}
 Let $G_1\longrightarrow G_0 \longrightarrow A\dashrightarrow$ be an $\Mbe$-triangle in $\xi$ with $G_0,G_1\in\QGP$, then there exist two $\Mbe$-triangles \[P_0\longrightarrow G'_0 \longrightarrow A\dashrightarrow \text{\quad and\quad} G'_1\longrightarrow P_1 \longrightarrow A\dashrightarrow \] 
 in $\xi$ with $P_0,P_1\in\P$ and  $G'_0,G'_1\in\QGP$. 
\end{cor}

\begin{thm}\label{HFY}
	Let $M\in  \widehat{\mathcal{QGP}}(\xi)$ and $n$ be a  positive integer. Then the following are equivalent:

	(1) $\Qgpd{M}\leq n$.
	
	(2) $\ext^{n+i}(M,Q)=0$ for $\forall i\geq 1$, $\forall Q\in\QP$.
	
	(3) For any integer $i$ with $0 \leq i \leq n$, there exists a $\xi$-exact complex
	\[
		\xymatrix{
			G_n\ar[r]&G_{n-1}\ar[r]&\cdots\ar[r]&G_2\ar[r]&G_1\ar[r]&G_0\ar[r]&M
	}
	\]
 such that $G_i\in\QGP$ and other $G_i\in\P$.

 Moreover, we have
 \[
\Qgpd{M}=\text{sup}\left\{m\in\mathbb{N}\ |\  \exists Q\in \QP \text{\ such that\ }\ext^m(M,Q)\neq 0 \right\}.
\]
\end{thm}
\begin{proof}
	$(1)\Rightarrow (2)$: Assume that  $\Qgpd{M}\leq n$, then there exists a $\xi$-exact  complex 
	\[
		\xymatrix{
			G_n\ar[r]&G_{n-1}\ar[r]&\cdots\ar[r]&G_2\ar[r]&G_1\ar[r]&G_0\ar[r]&M
	}
	\] 
with $G_i\in \QGP$  for  $0\leq i\leq n$ and  the $\xi$-resolution $\Mbe$-triangles $K_{j+1}\longrightarrow G_j \longrightarrow K_j\dashrightarrow$  (set $K_0=M, K_n=G_n$) for $0\leq j\leq n-1$. Let $Q$ be any object in $\QP$. By Lemma \ref{LZHL}, there is a long exact sequence
			\[
				\xymatrix@C=.5cm{
					\cdots\ar[r]&\ext^m(G_j,Q)\ar[r]& \ext^m(K_{j+1},Q)\ar[r]& \ext^{m+1}(K_j,Q)\ar[r]&\ext^{m+1}(G_j,Q)\ar[r]&\cdots
					}
			\] 
 where $\ext^m(G_j,Q)=\ext^{m+1}(G_j,Q)=0$ since Lemma \ref{TTTT}, so $\ext^m(K_{j+1},Q)\simeq \ext^{m+1}(K_j,Q)$, $m\geq 1$. By dimension shifting, for any $i\geqslant 1$, 
\[
	\ext^{n+i}(M,Q) \simeq \ext^i(K_{n},Q)=\ext^i(G_{n},Q)=0.
\]

$(2)\Rightarrow (1)$:  Assume that  $\ext^{n+i}(M,Q)=0$ for $\forall i\geq 1$, $\forall Q\in\QP$. Since $M\in  \widehat{\mathcal{QGP}}(\xi)$, there exists a $\xi$-exact complex
\[
		\xymatrix{
			G_m\ar[r]&G_{m-1}\ar[r]&\cdots\ar[r]&G_2\ar[r]&G_1\ar[r]&G_0\ar[r]&M
	}
	\] 
for some integer $m$, where $G_i\in\QGP$  and the $\xi$-resolution $\Mbe$-triangles are $K_{i+1}\longrightarrow G_i \longrightarrow K_i\dashrightarrow$  (set $K_0=M, K_m=G_m$) for any  $0\leq i\leq n-1$.  We only need to consider $m>n$. By the similar proof of ``only if''part, one can obtain that $\ext^i(K_{n+k},Q)\simeq \ext^{n+i+k}(M,Q)=0$ for any $i\geq 1$, $k\geq 0$. Applying Proposition \ref{BYSJ} on the $\Mbe$-triangle $G_m\rightarrow G_{m-1}\rightarrow K_{m-1}\dashrightarrow$ in $\xi$, implies that $K_{m-1}$ lies in $\QGP$. By repeating this on the other $\xi$-resolution $\Mbe$-triangles, one can deduce that $K_n\in\QGP$, and so $\Qgpd{M}\leq n$.

$(3)\Rightarrow (1)$: It is Obviously.

$(1)\Rightarrow (3)$: We  proceed by induction on $n$. Let $n=1$, the assertion is true by Corollary \ref{ZXR}. Now suppose that $n\geq 2$.  Then there exists a $\xi$-exact complex
\[
		\xymatrix{
			G^M_n\ar[r]&G^M_{n-1}\ar[r]&\cdots\ar[r]&G^M_2\ar[r]&G^M_1\ar[r]&G^M_0\ar[r]&M
	}
\]
with $G^M_i\in\QGP$  for any $0\leq i\leq n$ and the $\xi$-resolution $\Mbe$-triangles $K^M_{j+1}\longrightarrow G^M_j \longrightarrow K^M_j\dashrightarrow$  (set $K^M_0=M, K^M_n=G_n$) for any $0\leq j\leq n-1$. So there exists a $\xi$-exact complex 
\[
K^M_2\longrightarrow G^M_1 \longrightarrow G^M_0 \longrightarrow M.
\] 
It follows from Lemma \ref{YMY}, we get a $\xi$-exact complex $K^M_2\longrightarrow G'^M_1 \longrightarrow P^M_0 \longrightarrow M$ with $\xi$-resolution $\Mbe$-triangles 
\[
K^M_2\longrightarrow G'^M_1 \longrightarrow K'^M_1\dashrightarrow\text{\quad and\quad} K'^M_1\longrightarrow P^M_0 \longrightarrow M\dashrightarrow
\] 
where $G'^M_1\in\QGP$, $P^M_0\in\P$. So we get  a $\xi$-exact complex
\[
		\xymatrix{
			G^M_n\ar[r]&G^M_{n-1}\ar[r]&\cdots\ar[r]&G^M_2\ar[r]&G'^M_1\ar[r]&K'^M_1.
	}
\]
Note that $\Qgpd{K'^M_1}\leq n-1$, by the induction hypothesis, there exists a $\xi$-exact complex
\[
		\xymatrix{
			G_n\ar[r]&G_{n-1}\ar[r]&\cdots\ar[r]&G_2\ar[r]&G_1\ar[r]&K'^M_1
	}
\]
which one of $G_k\in\QGP$ for $ 1\leq k\leq n$ and other $G_i\in\P$.  Pasting the above $\xi$-exact complex and the $\Mbe$-triangle $K'^M_1\longrightarrow P^M_0 \longrightarrow M\dashrightarrow$ together, we obtain a $\xi$-exact complex
\[\xymatrix{
			G_n\ar[r]&G_{n-1}\ar[r]&\cdots\ar[r]&G_2\ar[r]&G_1\ar[r]&G_0\ar[r]&M,\quad(\text{set\ } G_0=P^M_0).
	}
\]
such that one of $G_k\in\QGP$ for $1\leq k \leq n$ and other $G_i\in\P$. 

Now we prove the case $k=0$. There exists an $\Mbe$-triangle $K\longrightarrow G \longrightarrow M\dashrightarrow$ in $\xi$ with $G\in\QGP$ and $\Qgpd{K}\leq n-1$ since $\Qgpd{M}\leq n$. By the induction hypothesis, there exists a $\xi$-exact complex
\[\xymatrix{
			G'_n\ar[r]&G'_{n-1}\ar[r]&\cdots\ar[r]&G'_2\ar[r]&G'_1\ar[r]&K
	}
\]
with $\xi$-resolution $\Mbe$-triangle $K_2\longrightarrow G'_1 \longrightarrow K\dashrightarrow$, where $G'_1\in\QGP$ and $G_i\in\P$ for all $2\leq i\leq n$. So there is  a $\xi$-exact complex $K_2\longrightarrow G'_1 \longrightarrow G\longrightarrow M$. By Lemma \ref{YMY}, we have a $\xi$-exact complex 
\[K_2 \longrightarrow P'_0 \longrightarrow G'_0 \longrightarrow M\]
 where $P'_0\in\P$ and $G'_0\in\QGP$.  Hence, there is a $\xi$-exact complex
\[\xymatrix{
			G'_n\ar[r]&G'_{n-1}\ar[r]&\cdots\ar[r]&G'_2\ar[r]&P'_0\ar[r]&G'_0\ar[r]&M
	}
\]
as desired.
\end{proof}

\begin{prop}
	Let $A,B\in\widehat{\mathcal{QGP}}(\xi)$ and  $A\longrightarrow B \longrightarrow C\dashrightarrow$ be an $\Mbe$-triangle in $\xi$ such that $A\neq 0$. If $C\in\QGP$, then $\Qgpd{A} = \Qgpd{B}$.
	\end{prop}
	\begin{proof}
		The result is clear from Theorem \ref{Dec31} if one of the $A$ or $B$ is in $\QGP$. 
		
		Assume that $\Qgpd{A}= n\geq 1$, then $\ext^{n+i}(A,Q)=0$  for any $i\geq 1$ by Theorem \ref{HFY}.  It follows from Lemma \ref{LZHL} that there is a long exact sequence
		\[
			\xymatrix@C=.5cm{
				\cdots\ar[r]&\ext^m(C,Q)\ar[r]& \ext^m(B,Q)\ar[r]& \ext^m(A,Q)\ar[r]&\ext^{m+1}(C,Q)\ar[r]&\cdots
				}
		\] 
	where $\ext^m(C,Q)=0$  since Lemma \ref{TTTT}, and so $\ext^m(B,Q)\simeq \ext^m(A,Q)$ for any $m\geq 1$. Then we have  $\ext^{n+i}(B,Q)=0$  for any $i\geq 1$, which implies that $\Qgpd{B}\leq n$ by Theorem \ref{HFY}. If  $\Qgpd{B}\leq n-1$, then $\ext^{n+i}(A,Q)\simeq \ext^{n+i}(B,Q)=0$ for any $i\geq 0$. Hence, $\Qgpd{A}\leq n-1$ by Theorem \ref{HFY} again which  is inconsistent with the assumption. Therefore, we have $\Qgpd{B}=n$. Similarly, if $\Qgpd{B}= n\geq 1$, we can get  $\Qgpd{A}= n$.
	\end{proof}

Dually, we have  

\begin{thm}\label{HSB}
	Let $M\in  \widehat{\mathcal{QIP}}(\xi)$ and $n$ be a  positive integer. Then the following are equivalent:

	(1) $\Qgid{M}\leq n$.
	
	(2) $\ext^{n+i}(Q,M)=0$ for $\forall i\geq 1$, $\forall Q\in\QI$.
	
	(3) For any integer $i$ with $0 \leq i \leq n$, there exists a $\xi$-exact complex
	\[
		\xymatrix{
			M\ar[r]&G_0\ar[r]&G_{-1}\ar[r]&G_{-2}\ar[r]&\cdots\ar[r]&G_{-(n-1)}\ar[r]&G_{-n}
	}
	\]
 such that $G_{-i}\in\QGI$ and other $G_{-i}\in\I$.

 Moreover, we have \[\Qgid{M}=\text{sup}\left\{m\in\mathbb{N}\ |\  \exists Q\in \QI \text{\ such that\ }\ext^m(Q,M)\neq 0 \right\}.\]
\end{thm}

{\small

}
\end{document}